\documentclass[12pt]{amsart}

\calclayout

\let\oldtocsection=\tocsection

\let\oldtocsubsection=\tocsubsection

\let\oldtocsubsubsection=\tocsubsubsection

\renewcommand{\tocsection}[2]{\hspace{0em}\oldtocsection{#1}{#2}}
\renewcommand{\tocsubsection}[2]{\hspace{5em}\oldtocsubsection{#1}{#2}}
\renewcommand{\tocsubsubsection}[2]{\hspace{2em}\oldtocsubsubsection{#1}{#2}}

\usepackage{graphicx}
 \usepackage{amssymb}

\usepackage[colorlinks=true, pdfstartview=FitV, linkcolor=blue, citecolor=blue, urlcolor=blue]{hyperref}

\def\Xint#1{\mathchoice
{\XXint\displaystyle\textstyle{#1}}%
{\XXint\textstyle\scriptstyle{#1}}%
{\XXint\scriptstyle\scriptscriptstyle{#1}}%
{\XXint\scriptscriptstyle\scriptscriptstyle{#1}}%
\!\int}
\def\XXint#1#2#3{{\setbox0=\hbox{$#1{#2#3}{\int}$}
\vcenter{\hbox{$#2#3$}}\kern-.5\wd0}}

\def\avgint{\Xint-}

\newtheorem{theorem}{Theorem}[section]
\newtheorem{lemma}[theorem]{Lemma}
\newtheorem{prop}[theorem]{Proposition}
\newtheorem{corollary}[theorem]{Corollary}

\newtheorem{definition}[theorem]{Definition}

\theoremstyle{definition}

\theoremstyle{remark}

\numberwithin{equation}{section}
 \allowdisplaybreaks

\def\ocirc#1{\ifmmode\setbox0=\hbox{$#1$}\dimen0=\ht0     \advance\dimen0 by1pt\rlap{\hbox to\wd0{\hss\raise\dimen0     \hbox{\hskip.2em$\scriptscriptstyle\circ$}\hss}}#1\else     {\accent"17 #1}\fi} 

\newcommand{\pp}{{p(\cdot)}}
\newcommand{\cpp}{{p'(\cdot)}}
\newcommand{\Lp}{L^{p(\cdot)}}

\newcommand{\R}{\mathbb R}
\newcommand{\N}{\mathbb N}
\newcommand{\Z}{\mathbb Z}

\newcommand{\subRn}{{{\mathbb R}^n}}
\newcommand{\Cc}{\mathbb C}

\newcommand{\D}{\mathcal D}
\newcommand{\Ss}{\mathcal S}

\newcommand{\F}{\mathcal F}

\newcommand{\X}{\mathfrak{X}}

\DeclareMathOperator{\re}{Re}

\DeclareMathOperator*{\essinf}{ess\,inf}
\DeclareMathOperator*{\esssup}{ess\,sup}

\DeclareMathOperator{\grad}{\nabla}

\newcommand{\loc}{\text{loc}}
\newcommand{\Rh}{\mathcal R}
\newcommand{\Div}{\text{div}}

\title{Extrapolation and Factorization}

\author{David Cruz-Uribe, OFS}
\address{Department of Mathematics, University of Alabama, Tuscaloosa, AL 35487}
\email{dcruzuribe@ua.edu}



\date{June 8, 2017}

\thanks{These notes are a modestly revised version of lecture notes that were
distributed to accompany my four lectures at the 2017 Spring School on Analysis at Paseky,
sponsored by Charles University, Prague, the Czech Republic, May 29 to
June 2, 2017.  I have
taken the opportunity to correct a number of typos and make some other
minor corrections to the text.     I want to tahnk the organizers for
inviting me to participate in the Spring School.  I also want to thank
Javier Mart{\'\i}nez, one of the participants, for carefully reading the printed version and
providing me with a long list of corrections. 
While writing these notes I was supported by NSF Grant DMS-1362425 and research funds from the
  Dean of the College of Arts \& Sciences, the University of Alabama.
}

\begin{document}

\maketitle

\vspace*{-0.25in}

 {\small \tableofcontents}

\section{Introduction}
\label{section:introduction}

The purpose of these lecture notes is to give an overview of the
theories of factorization and extrapolation for Muckenhoupt $A_p$
weights.  The $A_p$ weights were introduced by
Muckenhoupt~\cite{muckenhoupt72} in the early 1970s and a wide ranging
theory quickly developed:
see~\cite{duoandikoetxea01,dynkin-osilenker83,garcia-cuerva-rubiodefrancia85}
for details of this early history and extensive references.

Very early on the fine structure of $A_p$ weights--e.g. the $A_\infty$
condition, the reverse H\"older inequality and the fact that $A_p$
implies $A_{p-\epsilon}$--played an important role in the
theory.  These properties were central to the proofs of the boundedness of maximal
operators and singular integral operators on weighted spaces: see
Coifman and Fefferman~\cite{coifman-fefferman74}.

The deep structure revealed by the Jones factorization theorem--that every $A_p$ weight can be factored as the product of two $A_1$ weights--was conjectured by Muckenhoupt~\cite{muckenhoupt79} at the Williamstown  conference in 1979, and Jones~\cite{jones80} proved it at the same conference.   His proof was highly technical and was soon overshadowed by simpler approaches.  

A very simple proof of factorization was given by Coifman, Jones and
Rubio de Francia~\cite{coifman-jones-rubiodefrancia83}.  At the heart
of their proof were techniques developed by Rubio de Francia to prove
his own fundamental contribution to the theory of weighted norm
inequalities: the theory of
extrapolation~\cite{rubiodefrancia82,rubiodefrancia83,rubiodefrancia84}.
In its simplest form, this result says that if an operator $T$
satisfies
\[ \int_\subRn |Tf|^2 w\,dx \leq C\int_\subRn |f|^2w\,dx \]
for all weights $w\in A_2$, then for any $1<p<\infty$ and any $w\in A_p$
\[ \int_\subRn |Tf|^p w\,dx \leq C\int_\subRn |f|^pw\,dx. \]
Note in particular that this is true if we let $w=1$, so (unweighted)
$L^p$ estimates follow from weighted $L^2$ estimates.  In other words,
if a norm inequality holds at some point in a scale function spaces
(in this case weighted Lebesgue spaces), then it holds at every point
in this scale.  Early on, Antonio C\'ordoba~\cite{garcia-cuerva87}
summarized this by saying, ``{\em There are no $L^p$ spaces, only
  weighted $L^2$.}''

The theory of Rubio de Francia extrapolation (as it is now called) has
undergone a renaissance in the last twenty years.  New and simpler
proofs have been developed, including proofs that yield sharp
constants.  The theory has been extended to other settings and other
classes of weights, and has been used to prove norm inequalities in a large
class of Banach function spaces.  It has found a number of
applications, including the proof of the $A_2$ conjecture by
Hyt\"onen~\cite{hytonenP2010}.  Extrapolation has also been extended
to the setting of two weight norm inequalities.  The latter theory is
beyond the scope of our discussions here:
see~\cite{MR2797562,cruz-uribe-perez00} for further details.  But here
we want to note that it played a very surprising role in the disproof
of the long standing Muckenhoupt-Wheeden conjectures for singular
integral operators: see~\cite{CRV2012,MR3020857,reguera-thieleP}.

In these notes we  survey the theories of factorization and
extrapolation and we  describe some of the many applications.
They are organized as follows: in Section~\ref{section:weights} we
 define the $A_p$ weights and examine their close relationship
with the Hardy-Littlewood maximal operator.  We do so because the
maximal operator lies at the heart of the theories of factorization
and extrapolation, with the connection coming from the Rubio de Francia iteration
algorithm.  In Section~\ref{section:fine} we will consider the fine
properties of $A_p$ weights, and in particular we will prove the
reverse H\"older inequality. Somewhat surprisingly, though no longer
needed to prove the boundedness of the maximal operator and singular
integrals, the reverse H\"older inequality still plays an important
role in weighted theory.  In Section~\ref{section:factorization} we
prove the Jones factorization theorem and a generalization that shows
that the factorization also encodes information about the reverse
H\"older classes of weights.  Here we introduce the iteration
algorithm, which provides a tool for creating $A_1$ weights with very
precise control of their size.  In Section~\ref{section:extrapolation}
we prove the Rubio de Francia extrapolation theorem.  We adopt the
abstract perspective of families of extrapolation pairs which lets us
derive a number of corollaries as trivial consequences of the main
extrapolation theorem.  In Section~\ref{section:examples} we give three
applications of extrapolation; these have been chosen to illustrate
some of the typical ways in which extrapolation can be applied.  In
Section~\ref{section:variations} we discuss sharp constant
extrapolation, which is used to prove weighted inequalities with
optimal control of the constant in terms of the $A_p$ constant $[w]_{A_p}$.
We illustrate this by sketching an elementary proof of the $A_2$
conjecture and describing its application to regularity results for the
Beltrami operator.  In Section~\ref{section:restricted} we give two variants
of extrapolation which can be used to prove norm inequalities for a
restricted range of exponents.  Restricted range extrapolation arose
in the study of operators related to second order elliptic PDEs
and the Kato conjecture.  In Section~\ref{section:bilinear} we apply
restricted range extrapolation to prove a bilinear extrapolation
theorem.  Finally, in Section~\ref{section:BFS} we briefly discuss the
extension of Rubio de Francia extrapolation to other scales of Banach
function spaces, and in particular to the variable Lebesgue spaces.

In writing these notes there is a tension between brevity and
completeness, and in many instances brevity has won.  We provide
proofs of the central results on factorization and
extrapolation, and sketch many of the other proofs.  We provide
extensive references for the missing details and also for the
historical context in which these ideas were developed.  These notes
should be accessible to anyone who has completed a graduate course in
measure theory (say from Royden~\cite{MR1013117} or Wheeden and
Zygmund~\cite{MR3381284}), but some familiarity with the basics of
harmonic analysis (say the first six chapters of
Duoandikoetxea~\cite{duoandikoetxea01} or the first four chapters of
Grafakos~\cite{grafakos08a}) would be helpful.  An earlier set of
lecture notes~\cite{CruzUribe:2016ji} from a conference in Antequera,
Spain, in 2014  is a useful complement to the current document.
Though primarily concerned with fractional integral operators, it
contains a fairly complete and detailed treatment of one weight norm
inequalities from the perspective of dyadic operators.  We will make
extensive use of this ``dyadic technology'' in our applications.

\section{The maximal operator and Muckenhoupt $A_p$ weights}
\label{section:weights}

We begin with some basic definitions.  We will always be working on
$\R^n$ and the underlying measure will be Lebesgue measure.\footnote{Much of
what we say can be extended to the more general setting of spaces of
homogeneous type, but this is beyond the scope of these notes.} We
will denote this measure by $dx$, $dy$, etc.  The
variable $n$ will only be used to denote the dimension of the
underlying space.  By a weight $w$ we will mean a locally integrable,
non-negative function and we define $L^p(w)$, $1\leq p<\infty$, to be
$L^p(\R^n,w\,dx)$.   We will denote the set of bounded functions of
compact support by $L^\infty_c$, and the set of smooth functions of
compact support by $C_c^\infty$. 

By a cube we will always  mean a set of the form
\[ Q = [a_1,b_1)\times [a_2,b_2) \times \cdots \times [a_n,b_n), \]
where $b_j-a_j = \ell(Q)>0$ for $1\leq j \leq n$.  (In other words, we
consider cubes whose edges are parallel to the coordinate axes.)
Sometimes we will assume the cubes $Q$ are open and other times that
they are closed.  Since we will only be considering absolutely
continuous measures on $\R^n$, this will generally not matter and we
will take whatever is convenient.

We will work extensively with average integrals and  we will use the notation
\[ \avgint_Q w\,dx = \frac{1}{|Q|}\int_Q w\,dx.  \]
Though we will generally use this notation for cubes, it works equally
well if we replace the cube $Q$ by a measurable set $E$ such that $0<|E|<\infty$.
We will apply the same notation for averages with respect to other
(absolutely continuous) measures.  Given a weight $\sigma$ that is
positive a.e., define
\[ \avgint_Q w\,d\sigma = \frac{1}{\sigma(Q)} \int_Q w\sigma\,dx. \]

Constants will be denoted by $C$, $c$, etc. and may change value at
each appearance.  Generally, constants will depend on the dimension
$n$, the value $p$ of any associated $L^p$ space, and possibly the
operator under consideration.  For emphasis, we may denote this
dependence by writing $C(n,p)$, etc.  We will consider dependence on
the weight $w$ more carefully as we will make clear below.  If the
underlying constant is not particularly important, we may use the
notation $A\lesssim B$ to denote $A \leq cB$ for some constant $c>0$.

\medskip

We now define the fundamental weight classes we are interested in.

\begin{definition}
Given $1<p<\infty$, a weight $w$ is in the Muckenhoupt class $A_p$, denoted by $w\in A_p$, if 
$0<w(x)<\infty$ a.e. and 
\[ [w]_{A_p} = \sup_Q \left(\avgint_Q w\,dx\right)\left(\avgint_Q w^{1-p'}\,dx\right)^{p-1}<\infty, \]
where the supremum is taken over all cubes $Q$.
\end{definition}

Since $p'-1=\frac{p'}{p}$, we can also write the $A_p$ condition  in
an equivalent form using $L^{p}$ and $L^{p'}$ norms: for any cube $Q$,
\begin{equation} \label{eqn:Ap-norm}
  |Q|^{-1} \|w^{\frac{1}{p}}\chi_Q\|_p
  \|w^{-\frac{1}{p}}\chi_Q\|_{p'} \leq [w]_{A_p}^{\frac{1}{p}}.
\end{equation}

The definition of $A_p$ is symmetric:  given $w\in A_p$, let $\sigma =
w^{1-p'}$.  Then $\sigma \in A_{p'}$ and $[\sigma]_{A_{p'}} =[w]_{A_p}^{p'-1}$.  

To understand the $A_p$ condition, it is helpful to note that by H\"older's inequality, for every cube $Q$,
\[ 1 \leq \left(\avgint_Q w\,dx\right)\left(\avgint_Q w^{1-p'}\,dx\right)^{p-1}. \]
Thus, the $A_p$ condition can be thought of as a kind of ``reverse'' H\"older inequality.  

If we adopt the convention that $0\cdot \infty = 0$, then in this definition we could omit the assumption that $0<w(x)<\infty$ a.e.  However, nothing is gained by doing so, since this assumption is actually a consequence of the definition: see~\cite[Section~IV.1]{garcia-cuerva-rubiodefrancia85} for more details. 

\begin{definition}
When $p=1$, we say that a weight $w$ is in $A_1$, denoted by $w\in A_1$, if 
\[ [w]_{A_1} = \sup_Q \esssup_{x\in Q} w(x)^{-1} \avgint_Q w\,dy < \infty, \]
where again the supremum is taken over all cubes $Q$.   
\end{definition} 

Equivalently, $w\in A_1$ if for every cube $Q$,
\[ \avgint_Q w\,dy \leq [w]_{A_1} \essinf_{x\in Q} w(x), \]
or if $Mw(x) \leq [w]_{A_1} \essinf_{x\in Q} w(x)$, where $M$ denotes
the Hardy-Littlewood maximal operator (see below).  For a proof of
this equivalence,
see~\cite[Section~IV.1]{garcia-cuerva-rubiodefrancia85}.  The $A_1$
condition is the limit of the $A_p$ condition as $p\rightarrow 1$:  see
Rudin~\cite[pp.~73--4]{MR924157}.

By H\"older's inequality we have the following inclusions: for
$1<p<q<\infty$, $A_1\subset A_p \subset A_q$, and
$[w]_{A_q}\leq [w]_{A_p}\leq [w]_{A_1}$.  These inclusions are proper,
as is shown by the family of weights $w(x)=|x|^a$.  For $1<p<\infty$,
$w\in A_p$ if $-n<a<(p-1)n$, and $w\in A_1$ if $-n<a\leq 0$.  Define
the overarching class $A_\infty$ by
\[ A_\infty = \bigcup_{p\geq 1} A_p. \]

The weights in $A_\infty$ are characterized by a reverse Jensen inequality: there exists a constant $[w]_{A_\infty}$ such that for all cubes $Q$,
\[ \avgint_Q w\,dx \leq [w]_{A_\infty} \exp\left(\avgint_Q \log(w)\,dx\right). \]
For a proof, see~\cite[Section~IV.2]{garcia-cuerva-rubiodefrancia85}.
This inequality is the limit of the $A_p$ condition as
$p\rightarrow \infty$; consequently, we have that
$[w]_{A_\infty} \leq [w]_{A_p}$.  (Again, see
Rudin~\cite[p.~73]{MR924157}.)  But in fact,  for any
weight $w\in A_\infty$,
\[ [w]_{A_\infty} = \lim_{p\rightarrow \infty}[w]_{A_p}.  \]
For a proof, see Sbordone and Wik~\cite{MR1291957}.  

\bigskip

There is a close connection between the Muckenhoupt $A_p$ weights and
the Hardy-Littlewood maximal operator.  For $f\in L^1_\loc$ define
\[ Mf(x) = \sup_Q \avgint_Q |f|\,dy \cdot \chi_Q(x), \]
where the supremum is taken over all cubes $Q$.
It is well known that for $1\leq p<\infty$, $M$ satisfies the weak
$(p,p)$ inequality:  there exists $C>0$ such that for all $f$ and all $t>0$, 
\[ |\{ x\in \R^n : Mf(x)>t\}| \leq \frac{C}{t^p}\int_\subRn |f|^p\,dx; \]
further, for $1<p\leq \infty$ it satisfies the strong $(p,p)$
inequality: there exists $C>0$ such that for all $f$,
\[ \|Mf\|_p \leq C\|f\|_p. \]
The $A_p$ condition lets us prove the same inequalities in the
weighted Lebesgue spaces $L^p(w)$, $1\leq p<\infty$.

\begin{theorem} \label{thm:Ap-max}
Given $1\leq p < \infty$ and a weight $w$, the following are equivalent:
\begin{enumerate}

\item $w\in A_p$;

\item for all $t>0$,
\[ w(\{ x\in \R^n : Mf(x) > t \}) \leq C(n,p)[w]_{A_p} \frac{1}{t^p} \int_\subRn |f|^pw\,dx; \]

\item if in addition, $p>1$,
\[ \int_\subRn (Mf)^pw\,dx \leq C(n,p) [w]_{A_p}^{p'} \int_\subRn |f|^pw\,dx. \]
\end{enumerate}
\end{theorem}

For brevity, we will restrict ourselves to proving the equivalence of
(1) and (3) when $1<p<\infty$.  Furthermore, we will restrict
ourselves to the dyadic maximal operator.  Recall that the set of
dyadic cubes is the countable collection
\[ \Delta = \bigcup_{k\in \Z} \Delta_k, \]
where
\[ \Delta_k = \big\{ 2^{-k}\big([0,1)^n+m\big) :  m \in \Z^n \big\}. \]
The dyadic maximal operator is defined by
\[ M^df(x) = \sup_{Q\in \Delta} \avgint_Q |f|\,dy \cdot \chi_Q(x). \]
The proof we will give below can be adapted to the general case in
several ways; for this proof and for the proof of the weak type
inequality, we refer the reader
to~\cite{CruzUribe:2016ji,duoandikoetxea01,garcia-cuerva-rubiodefrancia85}.
We want to concentrate on the dyadic operator since it makes 
the main ideas of the proof clear while avoiding some technical difficulties.

The proof requires three lemmas.  The first is a construction that yields a collection of dyadic cubes often referred to as Calder\'on-Zygmund cubes.  For a proof, see~\cite{MR2797562,duoandikoetxea01,garcia-cuerva-rubiodefrancia85}.

\begin{lemma} \label{lemma:CZ-cubes}
Let $f\in L^p$, $1\leq p<\infty$.  Then for any $\lambda>0$, there exists a collection of pairwise disjoint dyadic cubes $\{Q_j\}$ such that
\[ \{ x\in \R^n : M^df(x)>\lambda \} = \bigcup_j Q_j \]
and
\[ \lambda < \avgint_{Q_j} |f|\,dx \leq 2^n\lambda. \]
Moreover, given  $a\geq 2^{n+1}$, for each $k\in \Z$ let $\{Q_j^k\}_j$ be the cubes gotten by taking $\lambda=a^k$.  Define
\[ \Omega_k = \{ x\in \R^n : M^df(x)> a^k \} = \bigcup_j Q_j^k, \]
and let $E_j^k = Q_j^k\setminus \Omega_{k+1}$.  Then the sets $E_j^k$ are pairwise disjoint and $|E_j^k|\geq \frac{1}{2}|Q_j^k|$.  
\end{lemma}

The second lemma shows that, in some sense, the measure $dw=w\,dx$ behaves like Lebesgue measure uniformly at all scales.

\begin{lemma} \label{lemma:Ap-prop}
Let $1\leq p<\infty$ and $w\in A_p$.  Then given any cube $Q$ and any measurable set $E\subset Q$, 
\[ \frac{|E|}{|Q|} \leq [w]_{A_p}^{\frac{1}{p}} \left(\frac{w(E)}{w(Q)}\right)^{\frac{1}{p}}. \]
\end{lemma}

\begin{proof}
When $p>1$, this follows at once from H\"older's inequality and the definition of $A_p$:
\begin{multline*}
\frac{|E|}{|Q|} = \avgint_Q \chi_E w^{\frac{1}{p}}w^{-\frac{1}{p}}\,dx
\leq \left(\avgint_Q w\chi_E\,dx\right)^{\frac{1}{p}}
\left(\avgint_Q w^{1-p'}\,dx\right) ^{\frac{1}{p'}} \\
\leq [w]_{A_p}^{\frac{1}{p}} \left(\avgint_Q w\chi_E\,dx\right)^{\frac{1}{p}}
\left(\avgint_Q w\,dx\right)^{-\frac{1}{p}}
= [w]_{A_p}^{\frac{1}{p}} \left(\frac{w(E)}{w(Q)}\right)^{\frac{1}{p}}.
\end{multline*}
When $p=1$ the proof follows directly from the definition of $A_1$. 
\end{proof}

For the third lemma, we introduce a weighted dyadic maximal operator.  Given a weight $\sigma$, let
\[ M^d_\sigma f(x) 
=  \sup_{Q\in \Delta} \avgint_Q |f|\,d\sigma \cdot \chi_Q(x). \]

\begin{lemma} \label{lemma:wtd-max-op}
Given a weight $\sigma$, then for all $1<p\leq \infty$,  there exists
a constant $C(p)>0$ such that for all $f$,  $\|M_\sigma^d f\|_p \leq C(p)\|f\|_p$. 
\end{lemma}

This inequality is proved exactly as the unweighted norm inequalities
for $M^d$.   When $p=\infty$ it is immediate.   When $p=1$, use
Lemma~\ref{lemma:CZ-cubes} to prove the weak $(1,1)$ inequality, and
then apply Marcinkiewicz interpolation to get the desired inequality.

\begin{proof}[Proof of Theorem~\ref{thm:Ap-max}]
As we indicated above, we will prove the equivalence of (1) and (3)
when $1<p<\infty$.  To prove necessity, fix a cube $Q$ and let $f=w^{1-p'}\chi_Q$.  Then for $x\in Q$, 
\[ M(w^{1-p'}\chi_Q)(x) \geq \avgint_Q w^{1-p'}\,dx, \]
and so by the strong type inequality,
\[ \left(\avgint_Q w^{1-p'}\,dx\right)^p \int_Q w\,dx \leq C\int_Q w^{1-p'}\,dx.  \]
The $A_p$ condition follows at once.

\medskip

To prove sufficiency we adapt a proof originally due to Christ and
Fefferman~\cite{MR684636}.  Let $\sigma=w^{1-p'}$.  By a standard
approximation argument, we may assume $f\geq 0$ and $f\in L_c^\infty$.
Fix $a\geq 2^{n+1}$.  Then, with the notation of Lemma~\ref{lemma:CZ-cubes}, we have that
\begin{align*}
\int_\subRn (M^d f)^pw\,dx
& = \sum_k \int_{\Omega_k\setminus \Omega_{k+1}}  (M^d f)^pw\,dx \\
& \lesssim \sum_k a^{kp} w(\Omega_k) \\
& = \sum_{k,j} a^{kp} w(Q_j^k) \\
& \leq \sum_{k,j} \bigg(\avgint_{Q_j^k} f\sigma^{-1}\sigma\,dx\bigg)^p w(Q_j^k) \\
& = \sum_{k,j} \bigg(\avgint_{Q_j^k} f\sigma^{-1}\,d\sigma\bigg)^p
\left(\avgint_{Q_j^k} w^{1-p'}\,dx\right)^{p-1}\avgint_{Q_j^k} w\,dx \, \sigma(Q_j^k); \\
\intertext{by Lemma~\ref{lemma:Ap-prop} applied to $\sigma \in A_{p'}$
  and by the definition of $A_p$,}
& \lesssim [w]_{A_p} [\sigma]_{A_{p'}} \sum_{k,j}
\bigg(\avgint_{Q_j^k} f\sigma^{-1}\,d\sigma\bigg)^p \sigma(E_j^k) \\
& \leq [w]_{A_p}^{p'} \sum_{k,j} \int_{E_j^k} M^d_\sigma (f\sigma^{-1})^p \,d\sigma \\
& \leq [w]_{A_p}^{p'} \int_\subRn M^d_\sigma (f\sigma^{-1})^p \,d\sigma; \\
\intertext{by Lemma~\ref{lemma:wtd-max-op},}
& \lesssim [w]_{A_p}^{p'} \int_\subRn (f\sigma^{-1})^p \,d\sigma \\
& = [w]_{A_p}^{p'}\int_\subRn f^pw\,dx. 
\end{align*}
\end{proof}

The constant we get in Theorem~\ref{thm:Ap-max} for the strong
$(p,p)$ inequality, in terms of the exponent on the
$A_p$ constant $[w]_{A_p}$,  is sharp: see Buckley~\cite{buckley93} for
examples.  Buckley also proved the strong $(p,p)$ inequality with this
constant using a different proof.  Yet another proof is due to
Lerner~\cite{lernerP}.  The fact that the sharp constant was implicit
in the proof of Christ and Fefferman~\cite{MR684636} seems to have been overlooked
for many years.\footnote{I learned this fact from Kabe Moen, who in
  turn learned it from an anonymous referee.}  The sharp constant for
the maximal operator plays a role in the proof of sharp constant
extrapolation discussed in Section~\ref{section:variations} below.

\section{The fine properties of $A_p$ weights}
\label{section:fine}

In this section we consider some of the fine properties of $A_p$
weights, particularly the reverse H\"older inequality, which yields
another characterization of the class $A_\infty$.

\begin{definition} \label{defn:rhi} Given a weight $w$ and $s>1$, we
  say that $w$ satisfies the reverse H\"older inequality with exponent
  $s$, denoted by $w\in RH_s$, if %
\[   [w]_{RH_s} = \sup_Q \left(\avgint_Q w^s\,dx\right)^{\frac{1}{s}}
\left(\avgint_Q w\,dx\right)^{-1} < \infty,  \]
where the supremum is taken over all cubes $Q$.
\end{definition}

\begin{theorem} \label{thm:rh-ineq}
If $w\in A_\infty$, then there exists $s>1$ such that $w\in RH_s$.  
In fact, there exists $s>1$ depending on $[w]_{A_p}$ such that for
every cube $Q$,
\[  \left(\avgint_Q w^s\,dx\right)^{\frac{1}{s}} \leq 2 \avgint_Q w\,dx. \]
Conversely, if $w\in RH_s$ for some $s>1$, then $w\in A_\infty$.  
\end{theorem}

We will only prove the first half of Theorem~\ref{thm:rh-ineq}.  For
the proof of the converse, which involves defining the $A_p$ and $RH_s$ classes
with respect to arbitrary measures and showing a certain ``duality''
condition, see~\cite[Section~IV.2]{garcia-cuerva-rubiodefrancia85}.

Before proving Theorem~\ref{thm:rh-ineq} we give two corollaries.  The first is important for historical reasons.

\begin{corollary} \label{cor:p-epsilon}
Given $1<p<\infty$, if $w\in A_p$, then there exists $\epsilon>0$ such that $w\in A_{p-\epsilon}$.
\end{corollary}

As a consequence of this corollary, the strong $(p,p)$ inequality follows
from the weak $(p,p)$ inequality by Marcinkiewicz interpolation: if
$w\in A_p$, then $w\in A_{p\pm \epsilon}$.  Moreover, by  a covering lemma
argument (using Lemma~\ref{lemma:CZ-cubes}) we can prove the weak
$(p\pm \epsilon, p\pm \epsilon)$ inequalities.  For this classical
approach, see~\cite{duoandikoetxea01,garcia-cuerva-rubiodefrancia85}.
The advantage of the proof of Theorem~\ref{thm:Ap-max} given above is
that it shows that the reverse H\"older inequality is not required.

\begin{proof}
Given $w\in A_p$, $w^{1-p'}\in A_{p'}\subset A_\infty$, so $w^{1-p'} \in RH_s$ for some $s>1$.  Fix $\epsilon>0$ such that 
\[ \frac{(p-\epsilon)'-1}{p'-1} = s.  \]
Then, given any cube $Q$,
\[ \left(\avgint_Q w^{1-(p-\epsilon)'}\,dx\right)^{(p-\epsilon)-1}
= \left( \avgint_Q \big(w^{1-p'}\big)^s\,dx\right)^{\frac{p-1}{s}}
\leq [w]_{RH_s} \left(\avgint_Q w^{1-p'}\,dx\right)^{p-1}; \]
it follows at once that $w\in A_{p-\epsilon}$.  
\end{proof}

The next corollary gives an inequality which is essentially the
opposite of that in Lemma~\ref{lemma:Ap-prop}.  Together, these two
results show that $A_p$ weights behave, in some sense, like constants uniformly at all
scales.

\begin{corollary} \label{cor:alt-Ainfty}
If $w\in A_\infty$, then there exist constants $C,\,\delta>0$ such that for any cube $Q$ and measurable set $E\subset Q$,
\[ \frac{w(E)}{w(Q)} \leq C\left(\frac{|E|}{|Q|}\right)^\delta. \]
\end{corollary}

\begin{proof}
This follows immediately from H\"older's inequality and the reverse H\"older inequality:  since $w\in RH_s$ for some $s>1$, 
\begin{multline*}
 w(E) = \int_Q w\chi_E \,dx
\leq \left(\avgint_Q w^s\,dx\right)^{\frac{1}{s}}|E|^{\frac{1}{s'}}|Q|^{\frac{1}{s}} \\
\leq [w]_{RH_s} \avgint_Q w\,dx |E|^{\frac{1}{s'}}|Q|^{\frac{1}{s}}
=  [w]_{RH_s}w(Q)\left(\frac{|E|}{|Q|}\right) ^{\frac{1}{s'}}. 
\end{multline*}
This gives the desired inequality with $C=[w]_{RH_s}$ and $\delta=\frac{1}{s'}$.
\end{proof}

The inequality in Corollary~\ref{cor:alt-Ainfty} is often taken as the
definition of the $A_\infty$ condition.  There are many equivalent
definitions: for a thorough treatment of them, see
Duoandikoetxea, Mart\'\i n-Reyes and Ombrosi~\cite{MR3473651}.

To prove the reverse H\"older inequality we need two lemmas.  The
first lets us replace an $A_p$ weight by its bounded truncation.

\begin{lemma} \label{lemma:truncate}
If $w\in A_p$, $1<p<\infty$, then for any $N>0$, $w_N=\min(w,N) \in A_p$ and 
$[w_N]_{A_p} \leq 2^p[w]_{A_p}$.
\end{lemma}

\begin{proof}
Since $w_N^{-1} \leq N^{-1}+w^{-1}$, and since
$(a+b)^{\frac{1}{p}}\leq a ^{\frac{1}{p}}+b ^{\frac{1}{p}}$, for any
cube $Q$, by Minkowski's inequality and~\eqref{eqn:Ap-norm},
\begin{multline*}
 \|w_N^{\frac{1}{p}}\chi_Q\|_p \|w_N^{-\frac{1}{p}}\chi_Q\|_{p'}
\leq \|N^{\frac{1}{p}}\chi_Q\|_p\|N^{-\frac{1}{p}}\chi_Q\|_{p'}+
\|w^{\frac{1}{p}}\chi_Q\|_p \|w^{-\frac{1}{p}}\chi_Q\|_{p'} \\
\leq |Q|+ [w]^{\frac{1}{p}}_{A_p}|Q| \leq 2
[w]^{\frac{1}{p}}_{A_p}|Q|. 
\end{multline*}
\end{proof}

The second lemma is a
local version of Lemma~\ref{lemma:CZ-cubes} that is proved in
exactly the same way.  Given a fixed cube $Q$, let $\Delta(Q)$ be the
set of all cubes that are gotten by bisecting  the
sides of $Q$, and then repeating this process inductively on each sub-cube so
formed.  For $x\in Q$ define the local dyadic maximal operator by 
\[ M^d_Qf(x) = \sup_{P\in \Delta(Q)} \avgint_P |f|\,dy \cdot \chi_P(x). \]

\begin{lemma} \label{lemma:CZ-cubes-local}
Given a cube $Q$, let $w$ be a weight such that $\avgint_Qw\,dx=1$.
Fix $a\geq 2^{n+1}$; then for each $k\geq 0$ we can write the set 
\[  \Omega_k = \{ x\in Q : M^d_Q w(x)> a^k \} = \bigcup_j Q_j^k, \]
where for each $k$ the cubes $Q_j^k \in \Delta(Q)$ are disjoint and satisfy
\[ a^k < \avgint_{Q_j^k} w\,dx \leq 2^n a^k. \]
Further, if $E_j^k = Q_j^k \setminus \Omega_{k+1}$, then the $E_j^k$
are pairwise disjoint and $|E_j^k|\geq \frac{1}{2}|Q_j^k|$.
\end{lemma}

\begin{proof}[Proof of Theorem~\ref{thm:rh-ineq}]
  Fix $w\in A_\infty$; we will assume for the moment that $w$ is
  bounded.  Fix a cube $Q$; by homogeneity, without loss of generality
  we may assume that $\avgint_Q w\,dx =1$.  Let $0<\epsilon<1$; we will
  fix the precise value below.  Then
\begin{align*}
\int_Q M_Q^d(w)^\epsilon w\,dx 
& = \int_0^\infty \epsilon t^{\epsilon-1} w(\{ x\in Q :
  M_Q^dw(x)>t\})\,dt \\
& = \int_0^1 \ldots + \int_1^\infty \ldots \\
& \leq w(Q) + \epsilon\sum_{k=0}^\infty w(\Omega_k)
  \int_{a^k}^{a^{k+1}} t^{\epsilon-1}\,dt \\
& \leq |Q| + \epsilon \sum_{k,j} a^{\epsilon(k+1)}w(Q_j^k)\int_{a^k}^{a^{k+1}} t^{-1}\,dt \\
& = |Q| + \epsilon a^\epsilon \log(a) \sum_{k,j} a^{k\epsilon}
  w(Q_j^k); \\
\intertext{by Lemma~\ref{lemma:Ap-prop},}
& \leq |Q| + \epsilon a^\epsilon \log(a) 2^p[w]_{A_p} \sum_{k,j} \bigg(\avgint_{Q_j^k} w\,dx\bigg)^\epsilon w(E_j^k) \\
& \leq |Q| +C(a,p)\epsilon \sum_{k,j} \int_{E_j^k} M_Q^d(w)^\epsilon w\,dx \\
& \leq |Q|+C(a,p)\epsilon \int_Q M_Q^d(w)^\epsilon w\,dx. 
\end{align*}
Now fix $\epsilon>0$ sufficiently small that
$C(\epsilon)=\frac{1}{2}$.  Since $w$ is bounded, 
\[ \int_{Q} M_Q^d(w)^\epsilon w\,dx <\infty.\]
Therefore, by  rearranging terms and by the Lebesgue differentiation
theorem we have that 
\[ \frac{1}{2}\int_Q w^{1+\epsilon}\,dx \leq \frac{1}{2}\int_Q M_Q^d(w)^\epsilon w\,dx \leq |Q|. \]
The desired inequality thus holds for bounded weights.  

Finally, given an
arbitrary weight $w$, by Lemma~\ref{lemma:truncate} and the previous
argument we have that
the reverse H\"older inequality holds for $w_N$ with a constant
independent of~$N$.  Hence, by the monotone convergence theorem it holds
for $w$.
\end{proof}

It is possible to give a very sharp estimate of the exponent $s$.  To do so we need to introduce another condition equivalent to the $A_\infty$ condition.  We say that a weight $w$ satisfies the Fujii-Wilson $A_\infty$ condition if 
\[ [w]_{A_\infty}' = \sup_Q w(Q)^{-1}\int_Q M(w\chi_Q)\,dx < \infty,  \]
where the supremum is taken over all cubes $Q$.  This condition is
equivalent to $w\in A_\infty$, a fact discovered independently by
Fujii~\cite{MR0481968} and Wilson~\cite{MR883661}.  It has the
advantage that it is generally much smaller than the other $A_\infty$
constants: see Beznosova and Reznikov~\cite{MR3234812}.  Using this
definition, Hyt\"onen and P\'erez~\cite{MR3092729} showed that
\[ s= 1 + \frac{1}{c(n) [w]_{A_\infty}'}. \]
Our proof of Theorem~\ref{thm:rh-ineq} is adapted from theirs; it is somewhat simpler since we do not get the sharp constant.

\medskip

If $w\in A_\infty$, then there exist $1<p,\,s<\infty$ such that $w\in
A_p$ and $w\in RH_s$.  However, there is no direct connection between
these two exponents:  The example of power weights shows that given
any pair of $p,\,s$, there exists $w\in A_p \cap RH_s$.   However, as
the next result shows, there is a weaker connection. This proposition
will play a role in restricted range extrapolation:  see
Section~\ref{section:restricted} below.

\begin{prop} \label{prop:Ap-RHs}
Given $1<p,\,s<\infty$ and a weight $w$, $w\in A_p\cap RH_s$ if and only if $w^s\in A_q$, where $q=s(p-1)+1$.  
\end{prop}

\begin{proof}
Suppose first that $w\in A_p\cap RH_s$.  By the definition of $q$ we have that $p'-1=s(q'-1)$.  Hence, for any cube $Q$,
\begin{multline*}
\left(\avgint_Q w^s\,dx\right)\left(\avgint_Q w^{s(1-q')}\,dx\right)^{q-1} \\
\leq [w]_{RH_s}^s\left(\avgint_Q w\,dx\right)^{s} \left(\avgint_Q w^{1-p'}\,dx\right)^{s(p-1)}
\leq [w]_{RH_s}^s[w]_{A_p}^s. 
\end{multline*}
Thus, $w^s \in A_q$.  

Conversely, if $w^s\in A_q$, then essentially the same argument using H\"older's inequality instead of the reverse H\"older inequality shows that $w\in A_p$.  Moreover, again given any cube $Q$, by the definition of $A_q$ and H\"older's inequality,
\[ \avgint_Q w^s\,dx = \avgint_Q w^s\,dx\left(\avgint_Q w^{s(1-q')}\,dx\right)^{q-1}
\left(\avgint_Q w^{1-p'}\,dx\right)^{-s(p-1)}
\leq [w^s]_{A_q}\left(\avgint_Q w\,dx\right)^s.  \]
Hence, $w\in RH_s$.  
\end{proof}

\medskip

As a final application of the reverse H\"older inequality we will
prove a multilinear version.  This inequality will be used in
Section~\ref{section:bilinear} below when we consider weighted norm
inequalities for bilinear operators.  This result was first proved
in~\cite{cruz-uribe-neugebauer95} in the bilinear case.  Recently, a
simpler proof for the general, multilinear case was given
in~\cite{DCU-KM-2017}.  To simplify the presentation, we give this proof in the bilinear case.

\begin{prop} \label{prop:multilinear-rhi} Given
  $w_1,\,w_2\in A_\infty$, suppose $w_1\in RH_{s}$ and
  $w_2 \in RH_{s'}$ for some $1<s<\infty$.  Then there exists $C>0$
  such that for every cube $Q$,
\[ \left(\avgint_Q w_1^{s}\,dx\right)^{\frac{1}{s}}
\left(\avgint_Q w_2^{s'}\,dx\right)^{\frac{1}{s'}}
\leq C\avgint_Q w_1w_2\,dx. \]
\end{prop}

\begin{proof}
  Since $w_1,\,w_2 \in A_\infty$, by Proposition~\ref{prop:Ap-RHs},
  $w_1^s,\,w_2^{s'}\in A_\infty$.  Moreover, since the $A_p$ classes
  are nested, we may assume that they
  are both in $A_q$ for some $q>1$.  Therefore, again by
  Proposition~\ref{prop:Ap-RHs}, there exists $0<r<1$, such that
  $w_1^{rs},\,w_2^{rs'}\in A_2\cap RH_{\frac{1}{r}}$.  If we use
  these two conditions and then H\"older's inequality three times, we
  get that for every cube $Q$,
\begin{align*}
\left(\avgint_Q w_1^{s}\,dx\right)^{\frac{1}{s}}
\left(\avgint_Q w_2^{s'}\,dx\right)^{\frac{1}{s'}}
& \lesssim 
\left(\avgint_Q w_1^{rs}\,dx\right)^{\frac{1}{rs}}
\left(\avgint_Q w_2^{rs'}\,dx\right)^{\frac{1}{rs'}} \\
& \lesssim 
\left(\avgint_Q w_1^{-rs}\,dx\right)^{-\frac{1}{rs}}
\left(\avgint_Q w_2^{-rs'}\,dx\right)^{-\frac{1}{rs'}} \\
& \leq 
\left(\avgint_Q w_1^{-r}w_2^{-r}\,dx\right)^{-\frac{1}{r}} \\
& \leq 
\left(\avgint_Q w_1^{r}w_2^{r}\,dx\right)^{\frac{1}{r}} \\
& \leq
\avgint_Q w_1w_2\,dx. 
\end{align*}
\end{proof}

\section{Factorization}
\label{section:factorization}

In this section we prove the Jones factorization theorem.  At the
heart of the proof is the Rubio de Francia
iteration algorithm, which allows us, given an arbitrary weight $u$,
to construct an $A_1$ weight $w$ that is the ``same size'' as $u$
in a precisely specified way.    The iteration algorithm also plays a
central role in the proof of extrapolation as we will see in
Section~\ref{section:extrapolation} below.

\begin{theorem} \label{thm:rdf-algorithm}
Fix $1<p<\infty$ and $w\in A_p$.  For any non-negative function $h\in
L^p(w)$, define
\[ \Rh h(x) = \sum_{k=0}^\infty \frac{M^k h(x)}{2^k \|M\|_{L^p(w)}^k}, \]
where for $k>0$, $M^kh = M \circ \cdots \circ Mh$ denotes $k$ iterations of
the maximal operator  and $M^0h=h$.  Then:
\begin{enumerate}

\item $h(x) \leq \Rh h(x)$;

\item $\|\Rh h \|_{L^p(w)} \leq 2 \|h\|_{L^p(w)}$;

\item $\Rh h \in A_1$ and $[\Rh h]_{A_1} \leq 2\|M\|_{L^p(w)}$. 

\end{enumerate}
\end{theorem}

\begin{proof} 
If we take the first term in the sum, (1) is immediate.  To prove (2)
we apply Minkowski's inequality:
\[ \|\Rh h \|_{L^p(w)} \leq \sum_{k=0}^\infty 
\frac{\|M^k h\|_{L^p(w)}}{2^k\|M\|_{L^p(w)}^k}
\leq \sum_{k=0}^\infty  2^{-k}\|h\|_{L^p(w)} =  2 \|h\|_{L^p(w)}. \]
Finally, (3) holds since the maximal operator is subadditive:
\[ M(\Rh h)(x) \leq \sum_{k=0}^\infty \frac{M^{k+1}h(x)}{2^k
    \|M\|_{L^p(w)}^k}
\leq 2\|M\|_{L^p(w)} \Rh h(x). \]
\end{proof}

We note that the existence of an $A_1$ majorant for a function $h$ is,
somewhat surprisingly, linked to $h$ being an element of the set
$\bigcup_{p>1}L^p$.  For a precise description of this connection, see
Knese, McCarthy and Moen~\cite{Knese:2016bd}.  

\medskip

An important feature of the proof of Theorem~\ref{thm:rdf-algorithm} is that
we only use the fact that the underlying operator is the maximal
operator to prove that $\Rh h \in A_1$.  If we replace $M$ by a
positive, sublinear operator $S$ that is bounded on $L^p(w)$, then the
same proof yields (1) and (2) and the $A_1$-type property that
$S(\Rh h) \leq 2\|S\|_{L^p(w)} \Rh h$.  This simple
generalization lets us prove the Jones factorization
theorem.

\begin{theorem} \label{thm:jones}
For $1<p<\infty$, a weight $w$ is in $A_p$ if and only if there exist
$w_1,\, w_2 \in A_1$ such that $w=w_1w_2^{1-p}$.  
\end{theorem}

\begin{proof}
One direction is easy:  in~\cite{MR2797562} we dubbed this fact
``reverse factorization.''\footnote{Unfortunately, this terminology
  has not gained universal acceptance.}
  Fix $p$ and $w_1,\,w_2\in A_1$.  Then for
any cube $Q$ and a.e. $x\in Q$, 
\[ \avgint_Q w_i\,dy \leq [w_i]_{A_1} w_i(x), \qquad i=1,2.  \]
Let $w=w_1w_2^{1-p}$; then we have that
\begin{align*}
&  \avgint_Q w\,dx \left(\avgint_Q w^{1-p'}\,dx\right)^{p-1} \\
& \qquad \qquad  = \avgint_Q w_1w_2^{1-p}\,dx \left(\avgint_Q
  [w_1w_2^{1-p}]^{1-p'}\,dx\right)^{p-1} \\
& \qquad \qquad  \leq [w_1]_{A_1}[w_2]_{A_1}^{p-1}
\avgint_Q w_1\,dx \left(\avgint_Q w_2\,dx\right)^{1-p}
\left(\avgint_Q w_2\,dx \right)^{p-1}\left(\avgint_Q
  w_1\,dx\right)^{-1} \\
& \qquad \qquad  = [w_1]_{A_1}[w_2]_{A_1}^{p-1}. 
\end{align*}

\medskip

The  difficult direction is the converse.  Fix $w\in A_p$,
$1<p<\infty$, and let $q=pp'>1$.   Define the operator 
\[ S_1f (x) =
  w(x)^{\frac{1}{q}}M(f^{p'}w^{-\frac{1}{p}})(x)^{\frac{1}{p'}}. \]
Then $S_1$ is sublinear and $S_1 : L^q \rightarrow L^q$ since
\[ \int_\subRn (S_1 f)^q\,dx =
\int_\subRn M(f^{p'}w^{-\frac{1}{p}})^p w\,dx
\leq C[w]_{A_p}^{p'} \int_\subRn f^q\,dx.  \]
In particular, $\|S_1\|_{L^q} \lesssim [w]_{A_p}^{\frac{1}{p}}$. 
Similarly, let $\sigma=w^{1-p'} \in A_{p'}$ and define
\[ S_2f = \sigma^{\frac{1}{q}}
  M(f^p\sigma^{-\frac{1}{p'}})^{\frac{1}{p}}. \]
Then $S_2$ is sublinear, $S_2 : L^q \rightarrow L^q$, and
$\|S_1\|_{L^q} \lesssim
[\sigma]_{A_{p'}}^{\frac{1}{p'}}=[w]_{A_p}^{\frac{1}{p}}$

Define $S=S_1+S_2$  and form the Rubio de Francia iteration algorithm
\[ \Rh h(x) = \sum_{k=0}^\infty \frac{S^kh(x)}{2^k\|S\|_{L^q}^k}. \]
Then, by the proof of
Theorem~\ref{thm:rdf-algorithm}, $\Rh : L^q \rightarrow L^q$.  Fix any
non-zero function $h\in L^q$; then 
$\Rh h$ is finite almost everywhere.  Moreover, 
$S(\Rh h)(x) \leq 2\|S\|_{L^q} \Rh h(x)$.  In particular, we have that
\[ w^{\frac{1}{q}}M((\Rh h)^{p'} w^{-\frac{1}{p}})^{\frac{1}{p'}}
= S_1(\Rh h) \lesssim \Rh h. \]
Hence, if we let $w_2= (\Rh h)^{p'} w^{-\frac{1}{p}}$, then 
this inequality becomes $Mw_2 \lesssim w_2$, so $w_2\in A_1$.  
Similarly, if we repeat this argument with $S_2$ in place of $S_1$,
we get $w_1= (\Rh h)^p \sigma^{-\frac{1}{p'}} \in A_1$.   Moreover,
it is immediate that $w_1w_2^{1-p} =
w^{\frac{1}{p}}w^{\frac{1}{p'}}=w$. 
\end{proof}

We note that in the proof of factorization, the function $h$ is chosen
essentially arbitrarily.  It is an open question whether the choice of
$h$ can be used to optimise this factorization in some way.

\medskip

The factorization in Theorem~\ref{thm:jones} also encodes information about the
reverse H\"older class of the weight $w$.  The proof is fairly easy
and mostly requires reinterpreting the terms in the Jones
factorization theorem.  This
generalization was first proved in~\cite{cruz-uribe-neugebauer95}.  To
state it, we need to introduce the class $RH_\infty$, which is related
to the reverse H\"older classes $RH_s$ in a way that is analogous to
the relationship between the $A_1$ and $A_p$ classes.

\begin{definition} 
Given a weight $w$, we say $w\in RH_\infty$ if
\[ [w]_{RH_\infty} = \sup_Q \esssup_{x\in Q} w(x) 
\left(\avgint_Q w(y)\,dy\right)^{-1} < \infty, \]
where the supremum is taken over all cubes $Q$.
\end{definition}

From the definition we have that for every cube $Q$ and a.e. $x\in Q$,
\[ w(x) \leq [w]_{RH_\infty} \avgint_Q w\,dy. \]
Raising both sides to the power $s>1$ and integrating over $Q$ shows
that $RH_\infty \subset RH_s$.  

\begin{theorem} \label{thm:gen-factorization}
For $1<p,\,s<\infty$, given a weight $w$, $w\in A_p \cap RH_s$ if and
only if there exist weights $v_1,\,v_2$ such that $w=v_1v_2$,  $v_1\in A_1\cap
RH_s$ and $v_2\in A_p\cap RH_\infty$.
\end{theorem}

 For the proof of Theorem~\ref{thm:gen-factorization}
we need three lemmas.  The first extends
Proposition~\ref{prop:Ap-RHs} to $A_1$ weights.

\begin{lemma} \label{lemma:A1-RHs}
Given a weight $w$ and $s>1$, $w\in A_1\cap RH_s$ if and only if
$w^s\in A_1$. 
\end{lemma}

\begin{proof}
Suppose first that $w\in A_1\cap RH_s$.  Given any cube $Q$,
\[ \avgint_Q w^s\,dy \lesssim \left(\avgint_Q w\,dy\right)^s
\lesssim \essinf_{x\in  Q} w(x)^s. \]
Hence, $w^s\in A_1$.  

Conversely, suppose $w^s\in A_1$. Given any cube $Q$, by H\"older's
inequality,
\[ \avgint_Q w\,dy \leq \left(\avgint_Q w^s\,dy\right)^{\frac{1}{s}}
\lesssim \essinf_{x\in Q} w(x) 
\leq \avgint_Q w\,dy. \]
It follows at once that $w\in A_1\cap RH_s$. 
\end{proof}

The next two lemmas consider dilations of $A_1$ and $RH_\infty$
weights.

\begin{lemma} \label{lemma:A1-neg-power}
If $w\in A_1$, then for any $r>0$, $w^{-r} \in RH_\infty$. 
\end{lemma}

\begin{proof}
Fix a cube $Q$.  By H\"older's inequality with exponent $p=1+r$,
\[ 1 = \avgint_Q w^{\frac{1}{p'}}w^{-\frac{1}{p'}}\,dx
\leq \left(\avgint_Q w\,dy\right)^{\frac{r}{1+r}}
\left(\avgint_Q w^{-r}\,dx\right)^{\frac{1}{1+r}}. \]
If we combine this with the fact that  $w\in A_1$, we get that for a.e. $x\in Q$,
\[ w(x)^{-r} \lesssim \left(\avgint_Q w\,dy\right)^{-r}
\leq \avgint_Q w^{-r}\,dy. \]
Hence, $w^{-r} \in RH_\infty$. 
\end{proof}

\begin{lemma} \label{lemma:RHinfty:pos-power}
If $w\in RH_\infty$, then for any $r>0$, $w^r \in RH_\infty$.  
\end{lemma}

\begin{proof}
If $r>1$, this is follows from H\"older's inequality: for any
cube $Q$ and a.e. $x\in Q$,
\[ w(x)^r \lesssim \left(\avgint_Q w\,dy\right)^r
\leq \avgint_Q w^r\,dy. \]
If $r<1$, then, since $w\in A_\infty$, by
Proposition~\ref{prop:Ap-RHs}, $w^r \in RH_{1/r}$.  Hence, we can
repeat the above argument using the reverse H\"older inequality to
get that $w^r \in RH_\infty$. 
\end{proof}

Note that the analog of Lemma~\ref{lemma:A1-neg-power} is not true for
$RH_\infty$ weights.  Since $|x|^{-a} \in A_1$ for $0\leq a <n$, by
Lemma~\ref{lemma:A1-neg-power}, $w(x)=|x|^b \in RH_\infty$ for any $b>0$.  But if $b>n$,
then $w^{-1} \not\in A_1$ since it is not locally integrable. 

We also note in passing that  the fact that $A_\infty$ is closed under the
dilation $w^r$, $0<r<1$, seems to be particular to this
class. For instance, there exists a doubling weight (i.e. $w$ such
that $w(2Q)\leq w(Q)$ for all cubes $Q$) such that $w^r$ is not
doubling for any $0<r<1$.  See~\cite{CruzUribe:2001tr}.

\begin{proof}[Proof of Theorem~\ref{thm:gen-factorization}]
We first fix $v_1\in A_1\cap RH_s$ and $v_2\in A_p \cap RH_\infty$.
By Lemmas~\ref{lemma:A1-RHs} and~\ref{lemma:RHinfty:pos-power},
$v_1^s\in A_1$ and $v_2^s\in RH_\infty$. Then given any cube $Q$,
\[ \avgint_Q w^s\,dx \lesssim
\avgint_Q v_1^s\,dx \avgint_Q v_2^s\,dx
\lesssim \avgint_Q v_1^s\,dx \left(\avgint_Q v_2\,dx\right)^s
\lesssim \left(\avgint_Q v_1v_2\,dx\right)^s. \]
Thus, $w\in RH_s$.  Similarly, by Lemma~\ref{lemma:A1-neg-power},
$v_1^{1-p'}\in RH_\infty$ and $v_1,\,v_2\in A_p$, and so 
\begin{multline*}
 \avgint_Q v_1v_2\,dx \left(\avgint_Q
    [v_1v_2]^{1-p'}\,dx\right)^{p-1} \\
\lesssim \avgint_Q v_1\,dx \avgint_Q v_2\,dx
\left(\avgint_Q v_1^{1-p'}\,dx\right)^{p-1}
\left(\avgint_Q v_2^{1-p'}\,dx\right)^{p-1} \leq
[v_1]_{A_p}[v_2]_{A_p}. 
\end{multline*}
Thus $w\in A_p$. 

\medskip

To prove the converse, fix $w\in A_p\cap RH_s$.  Then by
Proposition~\ref{prop:Ap-RHs}, $w^s\in A_q$  with $q=s(p-1)+1$.  But
then by Theorem~\ref{thm:jones} there exist $w_1,\,w_2\in
A_1$ such that $w^s=w_1w_2^{1-q}$, or equivalently,
$w=w_1^{\frac{1}{s}}w_2^{1-p} = v_1v_2$.   By
Lemma~\ref{lemma:A1-RHs}, $v_1\in A_1\cap RH_s$, and again by
Theorem~\ref{thm:jones}  and Lemma~\ref{lemma:A1-neg-power},
$v_2\in A_p\cap RH_\infty$.
\end{proof}

\medskip

Finally, we note that the iteration algorithm and the Jones
factorization theorem can be extended to other settings.  For the
factorization of the one-sided weights $A_p^{\pm}$,
see~\cite{MR2797562,martin-reyes-ortegasalvador-delatorre90}.  For the
extension of factorization to pairs of positive operators and to the
two weight setting,
see~\cite{MR2797562}.  For reverse factorization for the 
variable $A_\pp$ weights (the analog of the Muckenhoupt weights in the
variable Lebesgue spaces~\cite{MR2927495}) see~\cite{dcu-wangP}.

\section{Rubio de Francia extrapolation}
\label{section:extrapolation}

In this section we state and prove the Rubio de
Francia extrapolation theorem.  Our approach to extrapolation is based
on the abstract formalism of families of extrapolation pairs.  This
approach was introduced (in passing)
in~\cite{cruz-uribe-perez00} and first fully developed
in~\cite{cruz-uribe-martell-perez04}.  (See also~\cite{MR2797562}.)
It was implicit from the beginning that in extrapolating from an
inequality of the form
\[ \int_\subRn |Tf|^{p}w\,dx \lesssim \int_\subRn |f|^pw\,dx \]
the operator $T$ and its properties (positive, linear, etc.) played no
role in the proof.  Instead, all that mattered was that there existed
a pair of non-negative functions $(|Tf|,|f|)$ that satisfied a given
collection of norm inequalities.   Therefore, the proof goes through
working with any pair $(f,g)$ of non-negative functions.   

As a consequence, other kinds of inequalities can be  proved using
extrapolation.  For example, if we take pairs of the form $(|Tf|,Mf)$,
where, for example, $T$ is a Calder\'on-Zygmund singular integral operator,
then we can prove Coifman-Fefferman type inequalities~\cite{coifman-fefferman74}:
\[ \int_\subRn |Tf|^{p}w\,dx \lesssim \int_\subRn (Mf)^pw\,dx. \]
This was one of the reasons that this approach was adopted
in~\cite{cruz-uribe-martell-perez04}.   We discuss this and other
examples in detail below.

Hereafter, we will adopt the following conventions.   A family of extrapolation pairs 
$\F$ will consist of pairs of non-negative, measurable functions
$(f,g)$ that are not equal to $0$ a.e.  When we
write an inequality of the form
\[ \int_\subRn f^{p}w\,dx \leq C\int_\subRn g^pw\,dx, \qquad
  (f,g)\in \F, \]
where $0<p<\infty$ and $w\in A_q$, $1\leq q\leq \infty$, we mean
that this inequality holds for all pairs $(f,g)\in \F$ such that
$\|f\|_{L^p(w)}<\infty$--i.e., that the left-hand side of the
inequality is finite.  We further assume that the constant $C$ can
depend on $\F$, $p$, $q$, $n$, and the $[w]_{A_q}$ constant of $w$,
but that it does not depend on the specific weight $w$.  Note the
assumption that $f,\,g$ are not identically 0 simply rules out trivial
norm inequalities: since $A_\infty$ weights are positive a.e., we have
that $\|f\|_{L^p(w)},\,\|g\|_{L^p(w)}>0$.  Otherwise, if $f=0$, then these
inequalities hold for any $g$, and if $g=0$, they only hold if $f=0$.

If this seems mysterious, it may help to think of the particular
family
\[  \F = \{ (|Tf|,|f|), f\in \mathcal{X} \},  \]
where $T$ is some operator we are interested in and $\mathcal{X}$ is
some ``nice'' family of functions:  $L_c^\infty$, $C_c^\infty$, etc.
We will return to this point in Section~\ref{section:examples} below
when we consider applications of extrapolation.

\begin{theorem} \label{thm:rdf-extrapolation}
Given a family of extrapolation pairs $\F$, suppose that for some
$p_0$, $1\leq p_0<\infty$, and every $w_0\in A_{p_0}$,
\begin{equation} \label{eqn:rdf-extrapol1}
 \int_\subRn f^{p_0}w_0\,dx \leq C\int_\subRn g^{p_0}w_0\,dx, \qquad
  (f,g)\in \F. 
\end{equation}
Then for every $p$, $1<p<\infty$, and every $w\in A_p$, 
\begin{equation} \label{eqn:rdf-extrapol2}
 \int_\subRn f^{p}w\,dx \leq C\int_\subRn g^pw\,dx, \qquad
  (f,g)\in \F. 
\end{equation}
\end{theorem}

In the statement of Theorem~\ref{thm:rdf-extrapolation} we want to call
attention to the fact that while we can start with an endpoint
inequality (i.e., with the assumption that $p_0=1$), we cannot use
Rubio de Francia extrapolation to prove an endpoint inequality:  we must assume $p>1$.
To see that this restriction is natural, note that the operator $M^2=M\circ M$
is bounded on $L^p(w)$, $1<p<\infty$, $w\in A_p$, but does not 
satisfy an unweighted weak $(1,1)$ inequality.   It is possible to prove endpoint
estimates using generalizations of the extrapolation theorem, but much stronger, two weight
hypotheses are required.  See~\cite[Section~8.3]{MR2797562}.

\medskip

\begin{proof}
Before giving the details of the proof, we first sketch the basic
ideas underlying it.  To prove~\eqref{eqn:rdf-extrapol2}
from~\eqref{eqn:rdf-extrapol1} we need to pass between $L^p$ and
$L^{p_0}$ inequalities.  To do this we will use duality and
H\"older's inequality.  The original proofs of extrapolation required
two cases, depending on whether $p_0<p$ or $p_0>p$; we avoid this by
first dualising to $L^1$ and then using H\"older's inequality.  (This
comes with a cost:  see the discussion of sharp constants in
Section~\ref{section:variations} below.)

Next, to apply~\eqref{eqn:rdf-extrapol1} we need to construct an
$A_{p_0}$ weight, using only that we have a weight in $A_p$.  Here we
will use the Rubio de Francia iteration algorithm to construct $A_1$
weights, and then use reverse factorization (the easy half of
Theorem~\ref{thm:jones}) to form the desired weight.  

\medskip

Fix $p$, $1<p<\infty$, and $w\in A_p$.  We begin with the iteration algorithms.
Since $w\in A_p$, $\sigma=w^{1-p'} \in A_{p'}$.  Therefore, by
Theorem~\ref{thm:rdf-algorithm} we can
define the two iteration algorithms
\[
 \Rh_1 h_1 = \sum_{k=0}^\infty \frac{M^k h_1}{2^k \|M\|_{L^p(w)}^k}, \qquad
 \Rh_2 h_2 = \sum_{k=0}^\infty \frac{M^k h_2}{2^k \|M\|_{L^{p'}(\sigma)}^k},
\]
which satisfy the following properties: \\

\begin{tabular}{l l  l l}
 ($A_1$) & $h_1(x)\le \Rh_1 h_1(x)$
& \qquad
($A_2$) & $h_2(x)\le \Rh_2 h_2(x)$ \\[6pt]
($B_1$) & $\|\Rh_1 h_1\|_{L^{p}(w)}\le 2\|h_1\|_{L^{p}(w)}$
&\qquad 
($B_2$) & $\|\Rh_2 h_2\|_{L^{p'}(\sigma)}\le 2\|h_2\|_{L^{p'}(\sigma)}$
\\[6pt]
($C_1$) & $[\Rh_1 h_1]_{A_{1}}\le 2\|M\|_{L^{p}(w)}$
&\qquad
($C_2$) & $[\Rh_2 h_2]_{A_{1}}\le 2\|M\|_{L^{p'}(\sigma)}$.
\end{tabular}\\

\medskip

We now define $h_1$.  
Fix  $(f,g)\in \F$ such that
$\|f\|_{L^p(w)}<\infty$.  We may also assume $\|g\|_{L^p(w)}<\infty$,
since otherwise there is nothing to prove.    Define
\[ h_1 = \frac{f}{\|f\|_{L^p(w)}}+\frac{g}{\|g\|_{L^p(w)}}; \]
then $h_1 \in L^p(w)$ and $\|h_1\|_{L^p(w)}\leq 2$.   

We now prove the desired inequality.  
We will assume $1<p_0<\infty$;
the case $p_0=1$ requires some minor modifications to the argument and
we omit the details.  Since $f\in L^p(w)$, there exists a non-negative function
$h_2\in L^{p'}(w)$, $\|h_2\|_{L^{p'}(w)}=1$, such that 
\begin{align*}
\|f\|_{L^p(w)} 
& = \int_\subRn fh_2w\,dx. \\
\intertext{By ($A_2$) and H\"older's inequality, }
& \leq \int_\subRn f (\Rh_1 h_1)^{-\frac{1}{p_0'}}(\Rh_1
  h_1)^{\frac{1}{p_0'}}
\Rh_2 (h_2w)\,dx \\
& \leq \left(\int_\subRn f^{p_0}(\Rh_1 h_1)^{1-p_0}\Rh_2 (h_2w)\,dx \right)^{\frac{1}{p_0}}
\left(\int_\subRn \Rh_1 h_1 \Rh_2 (h_2w)\,dx
  \right)^{\frac{1}{p_0'}} \\
& = I_1^{\frac{1}{p_0}}\cdot I_2^{\frac{1}{p_0'}}.
\end{align*}

We first estimate $I_2$:  by $(B_1)$ and $(B_2)$, 
\begin{multline*}
I_2 
 = \int_\subRn \Rh_1 h_1 w^{\frac{1}{p}}\Rh_2 (h_2w) w^{-\frac{1}{p}}\,dx 
 \leq \|\Rh_1 h_1\|_{L^p(w)}\|\Rh_2(h_2w)\|_{L^{p'}(\sigma)} \\
 \leq 4\|h_1\|_{L^p(w)}\|h_2w\|_{L^{p'}(\sigma)} 
\leq 8\|h_2\|_{L^{p'}(w)} 
= 8.
\end{multline*}

To estimate $I_1$ we want to apply~\eqref{eqn:rdf-extrapol1}.  To do
so, first note that by $(C_1)$, $(C_2)$ and
Theorem~\ref{thm:jones},
\[ w_0 = (\Rh_1 h_1)^{1-p_0}\Rh_2 (h_2w) \in A_{p_0}. \]
Further, we have that $I_1<\infty$: by $(A_1)$, 
\[ \frac{f}{\|f\|_{L^p(w)}} \leq h_1 \leq \Rh_1 h_1, \]
and so 
\[ I_1 \leq \|f\|_{L^p(w)}^{p_0} \int_\subRn \Rh_1 h_1 \Rh_2
  (h_2w)\,dx < \infty. \]
Therefore, by~\eqref{eqn:rdf-extrapol1} and since, again by $(A_1)$,
\[ \frac{g}{\|g\|_{L^p(w)}} \leq h_1 \leq \Rh_1 h_1, \]
\[
I_1
 \lesssim \int_\subRn g^{p_0} (\Rh_1 h_1)^{1-p_0}\Rh_2 (h_2w)\,dx 
 \leq \|g\|_{L^p(w)}^{p_0} \int_\subRn \Rh_1 h_1 \Rh_2
  (h_2w)\,dx 
 \lesssim \|g\|_{L^p(w)}^{p_0}. 
\]
Combining these estimates we get~\eqref{eqn:rdf-extrapol2} and this
completes the proof.
\end{proof}

\bigskip

We will now prove three extensions of Rubio de Francia extrapolation
that are immediate consequences of Theorem~\ref{thm:rdf-extrapolation}
and the formalism of extrapolation pairs.

\begin{corollary} \label{cor:weak-rdf}
Given a family of extrapolation pairs $\F$, suppose that for some
$p_0$, $1\leq p_0<\infty$, and every $w_0\in A_{p_0}$,
\begin{equation} \label{eqn:weak-rdf1}
\|f\|_{L^{p_0,\infty}(w_0)} \leq C\|g\|_{L^{p_0}(w_0)}, \qquad (f,g)
\in \F. 
\end{equation}
Then for every $p$, $1<p<\infty$, and every $w\in A_p$,
\begin{equation} \label{eqn:weak-rdf2}
\|f\|_{L^{p,\infty}(w)} \leq C\|g\|_{L^{p}(w)}, \qquad (f,g)
\in \F. 
\end{equation}
\end{corollary}

\begin{proof}
Define a new family
\[ \F' = \big\{ (f_t,g) =\big( t\chi_{\{ x: f(x)>t\}}, g\big) : (f,g) \in \F,
  t>0 \big\}.  \]
Then by our assumption~\eqref{eqn:weak-rdf1},
\[ \|f_t\|_{L^{p_0}(w_0)} = tw_0(\{ x\in \R^n : f(x)>t\})^{\frac{1}{p_0}}
\leq \|f\|_{L^{p_0,\infty}(w_0)} \leq C\|g\|_{L^{p_0}(w_0)}. \]
Therefore,~\eqref{eqn:rdf-extrapol1} holds for  the family $\F'$.  Hence,
for all $p$ and $w\in A_p$,~\eqref{eqn:rdf-extrapol2} holds for $\F'$
with a constant independent of $t$,
and this implies that~\eqref{eqn:weak-rdf2} holds.
\end{proof}

Our second corollary shows that vector-valued inequalities are an
immediate consequence of Rubio de Francia extrapolation.  

\begin{corollary} \label{cor:vv-rdf}
Given a family of extrapolation pairs $\F$, suppose that for some
$p_0$, $1\leq p_0<\infty$, and every $w_0\in A_{p_0}$,
\begin{equation} \label{eqn:vv-rdf1}
\|f\|_{L^{p_0}(w_0)} \leq C\|g\|_{L^{p_0}(w_0)}, \qquad (f,g)
\in \F. 
\end{equation}
Then for every $1<p,\,q<\infty$ and every $w\in A_p$,
\begin{equation} \label{eqn:vv-rdf2}
\bigg\|\bigg(\sum_i f_i^q\bigg)^{\frac{1}{q}}\bigg\|_{L^p(w)}
\leq 
C \bigg\|\bigg(\sum_i g_i^q\bigg)^{\frac{1}{q}}\bigg\|_{L^p(w)}, 
\qquad \{(f_i,g_i)\} \subset \F. 
\end{equation}
\end{corollary}

\begin{proof}
Fix $q$, $1<q<\infty$, and define the new family of extrapolation pairs
\[ \F_q = \bigg\{ (F,G)= 
\bigg(\bigg(\sum_i f_i^q\bigg)^{\frac{1}{q}}, \bigg(\sum_i
g_i^q\bigg)^{\frac{1}{q}}\bigg)
: (f_i,g_i) \in \F \bigg\}, \]
where all of the sums are taken to be finite.
Since~\eqref{eqn:vv-rdf1} holds, by
Theorem~\ref{thm:rdf-extrapolation}, \eqref{eqn:rdf-extrapol2} holds
with $p=q$ and $w\in A_q$.  Hence, for all $(F,G)\in \F_q$,
\[ \|F\|_{L^q(w)}^q = \sum_i \int_\subRn f_i^q w\,dx
\lesssim \sum_i \int_\subRn g_i^q w\,dx = \|G\|_{L^q(w)}^q. \]
If we take this as our hypothesis, we can again apply
Theorem~\ref{thm:rdf-extrapolation} to conclude that for $1<p<\infty$
and $w\in A_p$, 
\[ \|F\|_{L^p(w)} \lesssim \|G\|_{L^p(w)},  \qquad (F,G) \in \F_q. \]
But this in turn is equivalent to~\eqref{eqn:vv-rdf2} for all finite
sums.  By the monotone convergence theorem we may pass to arbitrary
sums, which completes the proof.
\end{proof}

Our final corollary shows that we can rescale extrapolation families
and so derive the $A_\infty$ extrapolation theorem first proved
in~\cite{cruz-uribe-martell-perez04}. 

\begin{corollary} \label{cor:Ainfty-rdf}
Given a family of extrapolation pairs $\F$, suppose that for some
$p_0$, $0< p_0<\infty$, and every $w_0\in A_{\infty}$,
\begin{equation} \label{eqn:Ainfty-rdf1}
\|f\|_{L^{p_0}(w_0)} \leq C\|g\|_{L^{p_0}(w_0)}, \qquad (f,g)
\in \F. 
\end{equation}
Then for every $p$, $0<p<\infty$, and every $w\in A_\infty$,
\begin{equation} \label{eqn:Ainfty-rdf2}
\|f\|_{L^{p}(w)} \leq C\|g\|_{L^{p}(w)}, \qquad (f,g)
\in \F. 
\end{equation}
\end{corollary}

\begin{proof}
Fix $q_0$, $1<q_0<\infty$, and define the new family 
\[ \F_0 = \{ (F,G) = (f^{\frac{p_0}{q_0}},g^{\frac{p_0}{q_0}}) : (f,g)
  \in \F \}. \]
Then for every weight $w_0 \in A_{q_0}$ and every pair $(F,G)\in \F_0$, 
\[ \int_\subRn F^{q_0}w_0\,dx = \int_\subRn f^{p_0}w_0\,dx 
\lesssim \int_\subRn g^{p_0}w_0\,dx = \int_\subRn G^{q_0}w_0\,dx. \]
Therefore, \eqref{eqn:rdf-extrapol1} holds with $p_0=q_0$ for the
family $\F_0$, and so by Theorem~\ref{thm:rdf-extrapolation}, for any
$q$, $1<q<\infty$, and $w\in A_q$, 
$\|F\|_{L^q(w)} \lesssim  \|G\|_{L^q(w)}$, $(F,G)\in \F_0$.
Equivalently, 
\begin{equation} \label{eqn:Ainfty-rdf3}
 \int_\subRn f^{\frac{p_0}{q_0}q}w\,dx \lesssim 
\int_\subRn g^{\frac{p_0}{q_0}q}w\,dx,
\qquad (f,g) \in \F. 
\end{equation}
To complete the proof, we use that we can choose $q_0$ and $q$
freely.  Fix $0<p<\infty$ and $w\in A_\infty$.  Then $w\in A_q$ for
some $q>1$, and since the Muckenhoupt classes are nested, we may
assume that $q>\frac{p}{p_0}$.  Therefore, we can fix $q_0>1$ such
that $q=\frac{p}{p_0}q_0$, or $\frac{p_0}{q_0}q=p$.   Then
\eqref{eqn:Ainfty-rdf3} gives us~\eqref{eqn:Ainfty-rdf2}.  
\end{proof}

\section{Applications of Rubio de Francia extrapolation}
\label{section:examples}

In this section we give three applications of Rubio de Francia
extrapolation and the extensions proved in the last section.  These
examples are not exhaustive but should give some sense of the
ways in which extrapolation can be used.  

First, however, we consider further the technical hypothesis that we only
work with extrapolation pairs $(f,g)$ for which the left-hand side of
the weighted norm inequality in question is finite.   We can eliminate
this hypothesis with the following approximation argument.  Given a
family $\F$, we define a new family
\[ \F_0 = \{  (F,G) = (\min(f,N)\chi_{B(0,N)},g) : (f,g) \in \F, N \in \N
  \}.  \]
Since a weight $w\in A_\infty$ is locally integrable, we have that for
any $p$, $0<p<\infty$, and any pair $(F,G)\in \F_0$, 
\[ \int_\subRn F^p w\,dx \leq N^p w(B(0,N)) < \infty.  \]
Therefore, we can apply Theorem~\ref{thm:rdf-extrapolation} to the
family $\F_0$; the desired inequality for a given pair $(f,g)\in \F$,
whether or not $\|f\|_{L^p(w)}$ is finite, follows from the monotone
convergence theorem if we let $N\rightarrow \infty$.  

Given this reduction, it is now straightforward to prove weighted norm
inequalities for an operator $T$.  Suppose, for instance, that  for some $p_0\geq 1$
and $w_0\in A_{p_0}$ we know that 
\[ \|Tf\|_{L^{p_0}(w_0)} \lesssim \|f\|_{L^{p_0}(w_0)}, \]
where the constant depends only on $T$, $p_0$, $n$, $T$ and
$[w]_{A_{p_0}}$.  Then, in particular, it holds for some suitable
dense subset $\mathcal{X}$ of this space: e.g., $\mathcal{X}=L^\infty_c$, $C_c^\infty$, etc.
(Indeed, this inequality may only have been proved for functions in
this dense family.)  Then if we define the family of extrapolation
pairs 
\[ \F = \{ (|Tf|,|f|) :  f\in \mathcal{X} \}, \]
we have that the hypothesis~\eqref{eqn:rdf-extrapol1} of
Theorem~\ref{thm:rdf-extrapolation} holds, and so we can conclude that
for all $p$ and $w\in A_p$, \eqref{eqn:rdf-extrapol2} holds.  If we do
not know {\em a priori} that the left-hand side of this inequality is
finite, then we can apply the theorem to a family $\F_0$ defined as
above, and get the desired conclusion via approximation.   To prove
that the operator is bounded on all $f\in L^p(w)$, it suffices to use
another standard approximation argument. 

\medskip

We now turn to our examples.  The first is the well-known
vector-valued inequality for the maximal operator.  In the unweighted
case this was proved by Fefferman and Stein~\cite{fefferman-stein71};
the weighted estimate is due to Andersen and
John~\cite{andersen-john80}.   We want to emphasize that given the
scalar inequality in Theorem~\ref{thm:Ap-max}, the vector-valued inequality is an
immediate consequence of Corollary~\ref{cor:vv-rdf}:  no further work is required.

\begin{theorem} \label{thm:vv-max}
For every $1<p,\,q<\infty$ and every $w\in A_p$, 
\[  \bigg\|\bigg(\sum_i (Mf_i)^q\bigg)^{\frac{1}{q}}\bigg\|_{L^p(w)}
\lesssim
 \bigg\|\bigg(\sum_i |f_i|^q\bigg)^{\frac{1}{q}}\bigg\|_{L^p(w)}. \]
\end{theorem}

Similar vector-valued inequalities hold for other operators, such as
Calder\'on-Zygmund singular integral operators and commutators.  We
refer the reader to~\cite{cruz-uribe-martell-perez04,MR2797562} for
further examples.

Our second example uses extrapolation to prove the Coifman-Fefferman
inequality relating singular integrals and the maximal
operator~\cite{coifman-fefferman74}. 

\begin{theorem} \label{thm:coifman-fefferman}
Let $T$ be any Calder\'on-Zygmund singular integral operator.  Then
for $0<p<\infty$, $w\in A_\infty$ and $f\in L_c^\infty$,
\begin{equation} \label{eqn:cf-1}
\int_\subRn |Tf|^pw\,dx \lesssim \int_\subRn (Mf)^p w\,dx. 
\end{equation}
\end{theorem}

\begin{proof}
By Corollary~\ref{cor:Ainfty-rdf} it will suffice to prove~\eqref{eqn:cf-1}
when $p=1$.   We will sketch an easy proof in this case using the
theory of dyadic grids and sparse operators.  In the past decade, this
approach has come to play a central role in the theory of weighted
norm inequalities in harmonic analysis, starting with Hyt\"onen's
proof of the $A_2$ conjecture~\cite{hytonenP2010} (see
also~\cite{2014arXiv1409.4351C, MR3085756,lerner-nazarov14}).  For an
overview of these techniques (though from the perspective of fractional
integral operators) see~\cite{CruzUribe:2016ji}.

We begin by defining $3^n$ translates of the standard dyadic grid
using the so-called ``one-third'' trick:
\[ \D^t = \{ 2^j([0,1)^n + m + t) : j \in \Z, m \in \Z^n \}, \quad t
  \in \big\{0,\pm 1/3\big\}^n.  \]
The translation by $t$ does not affect any of the underlying
properties of the dyadic cubes.   In particular,
Lemmas~\ref{lemma:CZ-cubes} and~\ref{lemma:wtd-max-op} are still true,
in the latter replacing $M_\sigma^d$ with $M_\sigma^{\D^t}$, the dyadic maximal
defined with respect to cubes in $\D^t$.

A set $\Ss \subset \D^t$ is said to be
sparse if for every $Q\in \Ss$ there exists a measurable set $E_Q\subset Q$ such
that $|E_Q|\geq \frac{1}{2}|Q|$ and the sets $E_Q$ are pairwise
disjoint.    A sparse operator is a positive linear operator of the form
\[ T_\Ss f(x) = \sum_{Q\in \Ss} \avgint_Q f(y)\,dy \cdot\chi_Q(x). \]
These operators are dyadic models of Calder\'on-Zygmund singular
integrals.  More importantly, we have the following pointwise
estimate:  given a Calder\'on-Zygmund singular
integral $T$ and a function $f\in L_c^\infty$, there exist sparse
sets $\Ss_t\subset \D^t$ such that 
\begin{equation} \label{eqn:sparse-domination}
 |Tf(x)| \lesssim \sum_{t \in \{0,\pm1/3\}^n}
  T_{\Ss_t}(|f|)(x). 
\end{equation}
This estimate was originally proved by Lerner and
Nazarov~\cite{lerner-nazarov14} and independently by Conde-Alonso and
Rey~\cite{2014arXiv1409.4351C}.  Since then there have been a number
of new proofs and extensions:  see, for instance,
Lerner~\cite{MR3484688}, Hyt\"onen, {\em et
  al.}~\cite{MR3625128},  Lacey~\cite{Lacey:2015wf}, and
Conde-Alonso, {\em et al.}~\cite{diPlinio:2016ux}.

Given inequality~\eqref{eqn:sparse-domination}, to complete the proof
it will suffice to show that  given any sparse
set $\Ss\subset \D^t$ and $w\in A_\infty$, for non-negative $f\in
L^\infty_c$, 
\[ \int_\subRn T_\Ss f\, w\,dx \lesssim \int_\subRn Mf\, w\,dx; \]
in fact, we will prove this inequality with the Hardy-Littlewood
maximal operator replaced by the smaller dyadic maximal operator
$M^{\D^t}$ defined with respect to the cubes in $\D^t$.   But this is
almost trivial:  by Lemma~\ref{lemma:Ap-prop},
\begin{multline*}
\int_\subRn T_\Ss f\, w\,dx 
= \sum_{Q\in \Ss} \avgint_Q f\,dy \cdot w(Q)
\lesssim \sum_{Q\in \Ss} \avgint_Q f(y)\,dy \cdot w(E_Q) \\
 \leq \sum_{Q\in \Ss} \int_{E_Q} M^{\D^t}(f) w\,dx 
 \leq \int_\subRn M^{\D^t}(f) w\,dx. 
\end{multline*}
\end{proof}

For our final application we consider weighted norm inequalities for rough
singular integrals.  Unlike the previous results which were originally
proved without extrapolation, the following theorem was  proved
by Duoandikoetxea and Rubio de
Francia~\cite{duoandikoetxea-rubiodefrancia86} using extrapolation in
a critical way.  For a version of this result with quantitative
estimates on the constants, see~\cite{MR3625128}.  For a
generalization to a larger class of rough singular integrals,
see~\cite{diPlinio:2016ux}.

By a rough singular integral we mean the singular convolution operator
\[ T_\Omega f(x) = \text{p.v.}\int_\subRn 
\frac{\Omega(y/|y|)}{|y|^n}f(x-y)\,dy, \]
where $\Omega\in L^\infty(S^{n-1})$ and $\int_{S^{n-1}}\Omega\,dx =
0$. 

\begin{theorem} \label{thm:rough:sio}
Given a rough singular integral $T_\Omega$, for every $1<p<\infty$ and
every $w\in
A_p$, 
\begin{equation} \label{eqn:rough1}
\int_\subRn |T_\Omega f|^pw\,dx \lesssim \int_\subRn |f|^pw\,dx. 
\end{equation}
\end{theorem}

\begin{proof}
We sketch the argument in~\cite{duoandikoetxea-rubiodefrancia86},
emphasizing those parts of the proof that are more widely applicable.
We begin with the key reduction: by Theorem~\ref{thm:rdf-extrapolation} it suffices to
prove~\eqref{eqn:rough1} when $p=2$ and $w\in A_2$. 

Using Fourier transform techniques and Littlewood-Paley theory, they
showed that there exist operators $T_j$, $j\in \Z$, such that for all
$f\in L^2$,
\begin{equation} \label{eqn:Rdf-JD}
 T_\Omega f(x) = \sum_j T_jf(x). 
\end{equation}
Moreover, they showed that there exist $C,\,\alpha>0$ such that for all
$j$
\begin{equation} \label{eqn:unwtd}
  \|T_jf\|_2 \leq C2^{-\alpha |j|}\|f\|_2.  
\end{equation}
Thus, in particular, the series decomposition of $T_\Omega$ converges
in $L^2$.    

To get estimates in $L^2(w)$, $w\in A_2$, they used weighted
Littlewood-Paley theory~\cite{MR561835,MR1821243} to prove that for
all $f\in C_c^\infty$, 
\[ \|T_j f\|_{L^2(w)} \leq C \|f\|_{L^2(w)}, \]
where the constant $C>0$ is independent of $j$ and depends only
on $[w]_{A_2}$ and not on the weight itself.  However, the
constant has no decay, so this inequality cannot be used to directly prove
weighted norm inequalities for $T_\Omega$.  

To overcome this, note that since $w\in A_2$, $w^{-1} \in A_2$, so by the reverse
H\"older inequality (applied twice) there exists $\epsilon>0$ such that
$w^{1+\epsilon}\in A_2$, and in fact we can choose $\epsilon$ so that
$[w^{1+\epsilon}]_{A_2} \leq 4[w]_{A_2}$.  (See
Theorem~\ref{thm:rh-ineq}.)   Hence, for all $f\in C_c^\infty$, 
\begin{equation} \label{eqn:wtd}
 \|T_j f\|_{L^2(w^{1+\epsilon})} \leq C \|f\|_{L^2(w^{1+\epsilon})},
\end{equation}
and the constant is independent of $\epsilon$.
Therefore, by the interpolation with change of measure theorem due to
Stein and Weiss~\cite{stein-weiss58} (see also~\cite{MR0482275}) we
can interpolate between~\eqref{eqn:unwtd} and~\eqref{eqn:wtd} to get
\[ \|T_j f\|_{L^2(w)} \leq C2^{-\frac{\alpha\epsilon}{1+\epsilon}|j|} \|f\|_{L^2(w)}. \]
Hence, if we combine this with~\eqref{eqn:Rdf-JD}, we have that for all $w\in A_2$ and $f\in
C_c^\infty$,
\[ \|T_\Omega  f\|_{L^2(w)} \leq C \|f\|_{L^2(w)}, \]
which completes the proof.
\end{proof}

We want to highlight one feature of this proof. 
 The use of extrapolation to reduce the problem to proving $L^2$
estimates makes it possible to more easily prove various square
function and Littlewood-Paley estimates.  For an application of
this approach to multiplier theory and Kato-Ponce inequalities, see~\cite{CruzUribe:2016wv}.  For
an application in a somewhat different context, see Fefferman and
Pipher~\cite{fefferman-pipher97}.    

Further, by reducing the problem to $L^2$, the argument using
interpolation with change of measure allows unweighted inequalities
derived using Fourier transform estimates to be ``imported'' into
weighted $L^2(w)$, overcoming the fact that there are no useful
weighted estimates for the Fourier transform.  For another application
of this technique in the study of degenerate elliptic PDEs and
the Kato problem, see~\cite{cruz-riosP}.

\section{Sharp constant extrapolation}
\label{section:variations}

In this section we consider the problem of the sharp constant, in terms of the
$A_p$ constant, in Rubio de Francia extrapolation.  Suppose that we
know that for some $p_0$, $1\leq p_0<\infty$, and family of
extrapolation pairs $\F$, there exists a function $N_{p_0}$ such that
for every $w_0 \in A_{p_0}$, 
\[ \|f\|_{L^{p_0}(w_0)} \leq
  N_{p_0}([w]_{A_{p_0}})\|g\|_{L^{p_0}(w_0)}, \qquad (f,g) \in \F.  \]
Then for $1<p<\infty$ the problem is to find the optimal function  $N_p$ such
that for all $w\in A_p$,
\[ \|f\|_{L^{p}(w)} \leq
  N_{p}([w]_{A_{p}})\|g\|_{L^{p}(w)}, \qquad (f,g) \in \F.  \]
A close examination of the proof of
Theorem~\ref{thm:rdf-extrapolation} shows that we get
\begin{equation} \label{eqn:bad-exp}
 N_p ([w]_{A_{p}}) = c_1 N_{p_0}( c_2
  [w]_{A_p}^{1+\frac{p_0-1}{p-1}}), 
\end{equation}
where $c_1,\,c_2>0$ depend on $n,\,p,\,p_0$.    However, this can be
improved.

\begin{theorem} \label{thm:rdf-sharp}
Given $1\leq p_0<\infty$ and a family of extrapolation pairs $\F$,
suppose that for every $w_0\in A_{p_0}$,
\[ \|f\|_{L^{p_0}(w_0)} \leq
  N_{p_0}([w]_{A_{p_0}})\|g\|_{L^{p_0}(w_0)}, \qquad (f,g) \in \F.  \]
Then for  every $1<p<\infty$ and every $w\in A_p$, 
\[ \|f\|_{L^{p}(w)} \leq
  N_{p}([w]_{A_{p}})\|g\|_{L^{p}(w)}, \qquad (f,g) \in \F,  \]
where
\[ N_p ([w]_{A_{p}}) \leq C(p,p_0) N_{p_0}\big( C(n,p,p_0)
  [w]_{A_p}^{\max(1,\frac{p_0-1}{p-1})}\big). \]
\end{theorem}

As we will see below, this is the optimal result, since it yields
sharp inequalities for singular integrals and other operators.    For
a complete proof, see~\cite{duoandikoetxeaP}
or~\cite[Theorem~3.22]{MR2797562}.  Here we will restrict ourselves to
giving an idea of why
the proof of Theorem~\ref{thm:rdf-extrapolation} does not yield the
best constant, and how the proof has to be modified to achieve this.

One of the main features of the proof of
Theorem~\ref{thm:rdf-extrapolation} that distinguishes it from
previous proofs is that it only required a single case.  However, as a
consequence we have to use both iteration algorithms $\Rh_1$ and
$\Rh_2$.  Each one contributes a power of the $A_p$ constant of $w$,
so we get the sum $1+\frac{p_0-1}{p-1}$ in the exponent
in~\eqref{eqn:bad-exp}. 

To avoid this, we need to modify the proof and treat two cases.  If
$p<p_0$, then
we can   apply H\"older's inequality
immediately and then argue only using the iteration algorithm
$\Rh_2$.  This yields the exponent $\frac{p_0-1}{p-1}$.  On the other hand, if $p>p_0$,  then, instead of using
duality, we can fix $h_1$ so that $f\leq \Rh_1 h_1$ and write
\[ \int_\subRn f^p w\,dx = \int_\subRn f^{p_0} (\Rh_1
  h_1)^{-(p_0-p)}w\,dx.  \]
We can now modify the previous proof; this yields the exponent $1$.
In both cases we make use of the sharp constant in the weighted norm
inequalities for the maximal operator from Theorem~\ref{thm:Ap-max}. 

\bigskip

An interesting open question is to determine a sharp constant version
of Corollary~\ref{cor:Ainfty-rdf}, $A_\infty$ extrapolation.  The
precise constant may depend on which of the equivalent definitions of
$A_\infty$ is used.
\bigskip

We now want to consider two examples where the sharp constant, in
terms of the $[w]_{A_p}$ constant, matters.  The first is not a direct
application of Theorem~\ref{thm:rdf-sharp}, but it uses some
of the same ideas.   

\begin{prop}  \label{prop:sharp-p}
Let $T$ be an operator such that for some $p_0$, $1\leq p_0<\infty$,
and every $w_0\in A_{p_0}$, 
\[ \|Tf\|_{L^{p_0}(w_0)} \leq C[w_0]_{A_{p_0}}^{\alpha} \|f\|_{L^{p_0}(w_0)}.  \]
Then as $p\rightarrow \infty$, 
\[ \|Tf\|_p \leq Cp^\alpha \|f\|_p. \]
\end{prop}

\begin{proof}
Our proof uses the Rubio de Francia iteration algorithm and is, in
some sense, a special case of the proof of
Theorem~\ref{thm:rdf-sharp}.   Fix $p>p_0$ and define the iteration
algorithm
\[ \Rh h = \sum_{k=0}^\infty \frac{M^k h}{2^k \|M\|_{(p/p_0)'}^k}. \]
By the standard proofs of the boundedness of the maximal operator
(using Marcin\-kiewicz interpolation), 
$\|M\|_{(p/p_0)'} = C(n,p_0)p$.  Therefore, by
Theorem~\ref{thm:rdf-algorithm},
\[ [\Rh h]_{A_{p_0}} \leq [\Rh h]_{A_1} \leq 2 \|M\|_{(p/p_0)'} =
  C(n,p_0)p. \]
We can now argue as follows: by duality  there exists $h\in L^{(p/p_0)'}$,
$\|h\|_{(p/p_0)'}=1$, such that 
\begin{align*}
 \|Tf\|_p^{p_0}
&  = \int_\subRn |Tf|^{p_0} h\,dx; \\
\intertext{by the majorant property of $\Rh$, our hypothesis,
  H\"older's inequality and the boundedness of $\Rh$ on $L^{(p/p_0)'}$,}
& \leq \int_\subRn |Tf|^{p_0} \Rh h\,dx \\
& \leq C(n,p_0) p^{\alpha p_0} \int_\subRn |f|^{p_0} \Rh h\,dx \\
& \leq C(n,p_0) p^{\alpha p_0}\|f\|_p^{p_0}\|\Rh h\|_{(p/p_0)'} \\
& \leq C(n,p_0) p^{\alpha p_0}\|f\|_p^{p_0}. 
\end{align*}
\end{proof}

Proposition~\ref{prop:sharp-p} is implicit in Fefferman and
Pipher~\cite{fefferman-pipher97} who used it to get estimates for
multiparameter singular integrals.  In~\cite{dcu-martell-perez}
this argument was used to show that the exponent
$\alpha$ obtained for the weighted norm inequality for the dyadic
square function was the best possible.   Luque, P\'erez and
Rela~\cite{MR3342184} developed this idea further to show the general
relationship between the best exponent in the weighted inequalities
and the behavior of the constant in the unweighted inequality as
$p\rightarrow 1$ or $p\rightarrow\infty$.  

\medskip

A much deeper application of the optimal constant in extrapolation
comes from the study of the Beltrami equation in the plane.  Given a
bounded, open set $\Omega\subset \Cc$,  a
map $f : \Omega \rightarrow \Cc$ is a weakly $K$-quasiregular map if
$f\in W^{1,q}_{\loc}(\Omega)$, $1\leq q \leq 2$, and $f$ is a solution of
the Beltrami equation, 
\[ \partial_{\bar{z}} f(z) = \mu(z) \partial_z f(z), \qquad
\text{a.e. }z\in \Omega, \]
where $\mu$ is a bounded, complex-valued
function such that
\[ \|\mu\|_\infty \leq k = \frac{K-1}{K+1} <1.  \]
If $f$ is also continuous, then we say that it is $K$-quasiregular.
If $f\in W_{\loc}^{1,1+k+\epsilon}(\Omega)$, $\epsilon>0$, then it was
shown that $f$ is continuous; if $f\in
W_{\loc}^{1,1+k-\epsilon}(\Omega)$, then there are examples of weakly
$K$-quasiregular maps that are not $K$-quasiregular
(see~\cite{astala-iwaniec-saksman91}).  In the critical exponent case,
that is, when $f\in W_{\loc}^{1,1+k}(\Omega)$,  Astala, Iwaniec
and Saksman~\cite{astala-iwaniec-saksman91} showed that $f$ is continuous
if the Beurling-Ahlfors transform,
\[ Tf(z)  = \frac{1}{\pi}\int_\Cc \frac{f(w)}{(w-z)^2}dA(w), \]
satisfies a quantitative weighted norm inequality:  for every $p\geq
2$ there exists $C>0$ such that for every $w\in A_p$, 
\begin{equation} \label{eqn:BA}
 \|Tf\|_{L^p(w)} \leq C[w]_{A_p} \|f\|_{L^p(w)}.  
\end{equation}

The Beurling-Ahlfors transform is a two-dimensional Calder\'on-Zygmund
singular integral operator.   The original proofs of weighted norm
inequalities for singular integrals did not give quantitative bounds
in terms of the $A_p$ constant:  later, a close examination of the proofs
showed that the constant was on the order of $\exp(c[w]_{A_p})$.  
Buckley~\cite{buckley93} proved that for all $1<p<\infty$
and any singular integral $T$,
\[ \|Tf\|_{L^p(w)} \leq C[w]_{A_p}^{1+\frac{1}{p-1}} \|f\|_{L^p(w)};  \]
he also gave examples to show that in general, the smallest possible
exponent was $\max(1,\frac{1}{p-1})$.    

By Theorem~\ref{thm:rdf-sharp}, to prove that this is the sharp
exponent, and, in particular,  to prove \eqref{eqn:BA} for the
Beurling-Ahlfors transform, it suffices to prove that for $p=2$ and
$w\in A_2$,
\[ \|Tf\|_{L^2(w)} \leq C[w]_{A_2}\|f\|_{L^2(w)}.  \]
Because of this, the sharp constant problem for singular integrals
became known as the $A_2$ conjecture.

For the Beurling-Ahlfors transform, this conjecture was proved by Petermichl and
Volberg~\cite{petermichl-volberg02} using a Bellman function argument.
Petermichl then extended these techniques to prove it for the Hilbert
transform~\cite{MR2354322} and  the Riesz
transforms~\cite{petermichl08}.  A number of partial results were
obtained for more general singular integrals:  see, for
instance~\cite{dcu-martell-perez} and the references it contains.
The problem was finally solved in full generality by
Hyt\"onen~\cite{hytonenP2010}.  In all of these arguments
extrapolation played a central role in reducing to the case $p=2$.

The sparse domination inequality~\eqref{eqn:sparse-domination} was developed to simplify
the original argument of Hyt\"onen; here we give this proof.  

\begin{theorem} \label{thm:A2-conj}
Given a Calder\'on-Zygmund singular integral operator $T$, for every
$1<p<\infty$ and every $w\in A_p$, 
\[ \|Tf\|_{L^p(w)} \leq
  C[w]_{A_p}^{\max(1,\frac{1}{p-1})}\|f\|_{L^p(w)}. \]
\end{theorem}

\begin{proof}
  By Theorem~\ref{thm:rdf-sharp} and
  inequality~\eqref{eqn:sparse-domination}, it will suffice to show
  that if $\Ss$ is a sparse subset of some dyadic grid $\D^t$, then
  for all $w\in A_2$ and non-negative $f\in L_c^\infty$,
\[ \|T_\Ss f\|_{L^2(w)} \leq C[w]_{A_2}\|f\|_{L^2(w)}.  \]
To prove this we will use
an argument from~\cite{dcu-martell-perez}. 
Let $\sigma=w^{-1}$.  Then by duality there exists $h\in L^2(\sigma)$,
$\|h\|_{L^2(\sigma)}=1$, such that 
\begin{align*}
\|T_\Ss f\|_{L^2(w)} 
& = \int_\subRn T_\Ss f h\,dx \\
& = \sum_{Q\in \Ss} \avgint_Q f\,dx \avgint_Q h\,dx \, |Q|; \\
\intertext{by the definition of a sparse set and the definition of
  $A_2$,}
& \leq 2\sum_{Q\in \Ss}
\frac{w(Q)}{|Q|}\frac{\sigma(Q)}{|Q|} \,
\avgint_Q fw\,d\sigma \,\avgint_Q h\sigma \,dw \, |E_Q| \\
& \leq 2[w]_{A_2} \sum_Q \int_{E_Q} M^{\D^t}_\sigma(fw)
  M^{\D^t}_w(h\sigma)\,dx \\
&  \leq 2[w]_{A_2}\int_\subRn 
M^{\D^t}_\sigma(fw) \sigma^{1/2}
  M^{\D^t}_w(h\sigma)w^{1/2} \,dx; \\
\intertext{by H\"older's inequality and Lemma~\ref{lemma:wtd-max-op}
  (which holds for general dyadic grids with the same proof),}
&  \leq 2[w]_{A_2} \|M^{\D^t}_\sigma(fw) \|_{L^2(\sigma)}
\|M^{\D^t}_w(h\sigma)\|_{L^2(w)} \\
& \leq C[w]_{A_2} \|f\|_{L^2(w)}\|h\|_{L^2(\sigma)} \\
&  \leq C[w]_{A_2} \|f\|_{L^2(w)}.
\end{align*}
\end{proof}

Finally, we note in passing that it is possible to prove
Theorem~\ref{thm:A2-conj} without using extrapolation.  The $L^2$
estimate for sparse operators can be extended to weighted $L^p$,
though the resulting proof is more complicated.  See
Moen~\cite{moen-2012} for the details.

\section{Restricted range extrapolation}
\label{section:restricted}

In this section we consider a second variation of Rubio de Francia
extrapolation, restricted range extrapolation.
Restricted range extrapolation was first proved by Auscher and
Martell~\cite{auscher-martell07} (and also by
Duoandikoetxea~\cite{duoandikoetxea-moyua-oruetxebarria-seijoP} but
with a
very different perspective).  Auscher and Martell were considering
families of operators associated with certain second order elliptic
PDEs; these PDEs in turn were of
interest because of their connection with the Kato conjecture
(for a history of this problem,
see~\cite{auscher-hofmann-lacey-mcintosh-tchamitchian02} and the
references it contains).  Let $A$ be an $n\times n$ matrix of
measurable, complex valued functions that for some
$0<\lambda<\Lambda<\infty$ satisfies the ellipticity
conditions
\[ 
\lambda |\xi|^2  \leq \re \langle A\xi,\xi\rangle, \qquad
|\langle A\xi,\nu \rangle| \leq \Lambda |\xi||\nu|,
\qquad \xi,\,\nu \in \Cc^n.
\]
Define the differential operator $Lu=
-\Div A\grad u$.  Then the Kato conjecture states that  for all $u\in
W^{1,2}(\R^n)$ (i.e., $u$ such that $u,\,\grad u \in L^2$), 
\begin{equation} \label{eqn:kato}
 \|L^{1/2} u\|_2 \approx \|\grad u\|_2, 
\end{equation}
where the operator $L^{1/2}$ is defined using the functional
calculus.  We can define (again, via the functional calculus) the
associated Riesz transform $\grad L^{-1/2}$; when $A$ is the identity
matrix, this is just the classical (vector) Riesz transform.  
 It follows
from~\eqref{eqn:kato} that
\[ \|\grad L^{-1/2} u\|_2 \lesssim \|u\|_2.  \]
These operators also satisfy $L^p$ inequalities, $p\neq 2$, but unlike
the classical Riesz transforms, one cannot take $p\in
(1,\infty)$.  Rather, for each operator $L$ there exist $1\leq
p_-<2<p_+ \leq \infty$
such that if $p \in (p_-,p_+)$, then 
\[ \|\grad L^{-1/2} u\|_p \lesssim \|u\|_p.  \]
In certain cases this estimate holds for all $p\in (1,\infty)$, but there exist
operators such that $(p_-,p_+)=(2-\delta,2+\epsilon)$ where
$\epsilon,\,\delta>0$ are small:  see~\cite{auscher07}.

It is natural to ask under what conditions the corresponding weighted
inequalities,
\[ \|\grad L^{-1/2} u\|_{L^p(w)} \lesssim \|u\|_{L^p(w)}, \]
hold.  Auscher and Martell~\cite{auscher-martell06} showed that for
$p_-<p<p_+$, this inequality holds for all weights $w$ such that
$w\in A_{p/p_-}\cap RH_{(p_+/p)'}$, where we interpret $\infty'=1$.
(Note that by Theorem~\ref{thm:gen-factorization} this class is never empty.)
As part of the (lengthy) proof of this inequality, they proved a
restricted range extrapolation theorem.

\begin{theorem} \label{thm:restricted-range}
Given a family of extrapolation pairs $\F$, suppose there exist $1\leq
p_-<p_0<p_+\leq\infty$ such that for every $w_0\in A_{p_0/p_-}\cap
RH_{(p_+/p_0)'}$,
\[ \int_\subRn f^{p_0}w_0\,dx \lesssim \int_\subRn g^{p_0}w_0\,dx,
\qquad (f,g) \in \F. \]
Then for every $p_-<p<p_+$ and every $w\in A_{p/p_-}\cap
RH_{(p_+/p)'}$,
\[ \int_\subRn f^{p}w\,dx \lesssim \int_\subRn g^{p}w\,dx,
\qquad (f,g) \in \F. \]
\end{theorem}

We will not prove this theorem, as the proof is very long and
technical, and we refer the reader to~\cite[Theorem~3.31]{MR2797562}
for the details. 
Instead, we will describe the heuristic argument that leads to the
proof.  This approach was used to find many of the proofs in~\cite
{MR2797562} but was never made explicit and indeed, the traces were
generally removed.  A detailed explanation of it, in the context of
proving extrapolation in the variable Lebesgue spaces, was given
in~\cite[Section~4]{CruzUribeSFO:2017km}.  

To expand upon the discussion at the beginning of the proof of
Theorem~\ref{thm:rdf-extrapolation}, to prove Theorem~\ref{thm:restricted-range} we have the following at our disposal:
\begin{itemize}
\item The boundedness of the maximal operator on $L^q(w)$ when $w\in
  A_q$.  In this case, however, we will not take $q=p$ and $w\in
  A_p$.  By our hypothesis and Proposition~\ref{prop:Ap-RHs}, we have
  $u=w^{(p_+/p)'} \in A_\tau$, where 
\[ \tau = \left(\frac{p_+}{p}\right)'\left(\frac{p}{p_-}-1\right)+1
= \frac{\frac{1}{p_-}-\frac{1}{p}}{\frac{1}{p}-\frac{1}{p_+}} + 1. \]
Though the final expression looks more complicated, in retrospect
we believe that this is the correct way to write it:  see the
calculations in~\cite{CruzUribe:2017vx}.

\item Using the weights $u$ and $v=u^{1-\tau'}$ we can define Rubio de
  Francia iteration algorithms $\Rh_1$ and $\Rh_2$.   However, these are no longer bounded
  on the space $L^p(w)$ or its dual, so it is necessary to rescale.
  We do this by introducing functions of the form
\[ H_1  = \Rh_1(h_1^\alpha w^\beta)^{\frac{1}{\alpha}}
  w^{-\frac{\beta}{\alpha}}, \quad 
H_2 = \Rh_2(h_2^\gamma w^\delta)^{\frac{1}{\gamma}}
w^{-\frac{\delta}{\gamma}}.  \]

\item Finally, we can use duality, but dualising to $p=1$ may no
  longer work.  Therefore, we fix $1\leq s < \min(p,p_0)$ and dualize
  to $L^s$: for some $h_2\in L^{(p/s)'}(w)$, $\|h_2\|_{
    L^{(p/s)'}(w)}=1$, 
\[ \|f\|_{L^p(w)}^s = \int_\subRn f^s h_2 w\,dx.\]
\end{itemize}

Given these tools, the goal is to follow the proof of
Theorem~\ref{thm:rdf-algorithm}, writing
\[ \int_\subRn f^s h_2 w\,dx \leq 
\int_\subRn f^s H_1^{-\epsilon}H_1^\epsilon H_2w\,dx, \]
applying Holder's inequality, and then using
Theorem~\ref{thm:gen-factorization} to create a weight $W\in  
A_{p_0/p_-}\cap
RH_{(p_+/p_0)'}$.  At each stage this imposes constraints on the
constants $\alpha,\,\beta,\,\gamma,\,\delta$, $\epsilon$ and $s$, and it is the
``miracle'' of extrapolation that these constraints can all be
satisfied simultaneously.  

\bigskip

Very recently, Martell and I were interested in proving a bilinear
version of Theorem~\ref{thm:restricted-range}, with the goal of
proving weighted norm inequalities for the bilinear Hilbert
transform, generalizing a result of Culiuc, di Plinio and
Ou~\cite{Culiuc:2016wr}. (See Section~\ref{section:bilinear} below.)
Using an idea from Duoandikoetxea~\cite{duoandikoetxeaP} we showed
that we could prove the desired bilinear extrapolation theorem if we
could prove an off-diagonal version of
Theorem~\ref{thm:restricted-range}.

An
off-diagonal inequality is an inequality of the form $\|f\|_{L^q(w^q)}\lesssim
\|g\|_{L^p(w^p)}$, $p\neq q$; we write it in this way, with different powers on the
weight on the left and right-hand sides, in order to make the inequality homogeneous in the weight.
Off-diagonal inequalities are natural for operators such as the
fractional integral operator
\[ I_\alpha f(x) = \int_\subRn \frac{f(y)}{|x-y|^{n-\alpha}}\,dy,
  \qquad 0<\alpha <n. \]
Muckenhoupt and Wheeden~\cite{muckenhoupt-wheeden74} proved that for
$1<p<\frac{n}{\alpha}$ and $1<q<\infty$ such that
$\frac{1}{p}-\frac{1}{q}=\frac{\alpha}{n}$, a necessary and sufficient
condition for the inequality
\begin{equation} \label{eqn:MW1}
 \|I_\alpha f\|_{L^q(w^q)} \lesssim \|f\|_{L^p(w^p)}
\end{equation}
is that $w\in A_{p,q}$:
\begin{equation} \label{eqn:Apq-cond}
 [w]_{A_{p,q}} =
\sup_Q \left(\avgint_Q w^q\,dx\right)^{\frac{1}{q}}
\left(\avgint_Q w^{-p'}\,dx\right)^{\frac{1}{p}}<\infty,
\end{equation}
where the supremum is taken over all cubes $Q$.
When $p=q$, this is equivalent to assuming $w^p \in A_p$.  

\medskip

In~\cite{CruzUribe:2017vx} we proved the following limited range, off-diagonal extrapolation theorem.

\begin{theorem} \label{thm:restricted-offdiag}
Given $0\leq p_-<p_+\leq \infty$ and a family of extrapolation pairs
$\F$, suppose that for some $p_0,\,q_0 \in (0,\infty)$ such that
$p_-\leq p_0 \leq p_+$, $\frac{1}{q_0}-\frac{1}{p_0}+\frac{1}{p_+}\geq
0$, and all $w$ such that $w_0^{p_0} \in A_{p_0/p_-}\cap RH_{(p_+/p_0)'}$, 
\begin{equation} \label{eqn:restricted-offdiag1}
\left(\int_\subRn f^{q_0}w_0^{q_0}\,dx\right)^{\frac{1}{q_0}}
\lesssim 
\left(\int_\subRn g^{p_0}w_0^{p_0}\,dx\right)^{\frac{1}{p_0}},
\qquad (f,g) \in \F. 
\end{equation}
Then for every $p,\,q$ such that $p_-<p<p_+$, $0<q<\infty$,
$\frac{1}{p}-\frac{1}{q}=\frac{1}{p_0}-\frac{1}{q_0}$, and every $w$
such that  $w^{p} \in A_{p/p_-}\cap RH_{(p_+/p)'}$, 
\begin{equation} \label{eqn:restricted-offdiag2}
\left(\int_\subRn f^{q}w^{q}\,dx\right)^{\frac{1}{q}}
\lesssim 
\left(\int_\subRn g^{p}w^{p}\,dx\right)^{\frac{1}{p}},
\qquad (f,g) \in \F. 
\end{equation}
\end{theorem}

The proof of Theorem~\ref{thm:restricted-offdiag} is similar to that
of Theorem~\ref{thm:restricted-offdiag};  following the heuristic argument laid
out above, the central difficulty in the
proof is determining the constraints on the constants and showing
that they are consistent.  

Theorem~\ref{thm:restricted-offdiag} generalizes almost all of the extrapolation
theorems we have discussed as well as several others in the literature
we have passed over.
\begin{itemize}
\item If we take $p_-=1$, $p_+=\infty$, and $p_0=q_0$, then we get the
  classical Rubio de Francia extrapolation theorem,
  Theorem~\ref{thm:rdf-algorithm}.  

\item If we take $p_-=0$, $p_+=\infty$, and $p_0=q_0$, then we get
  $A_\infty$ extrapolation, Corollary~\ref{cor:Ainfty-rdf}.

\item If we take $p_-=1$, $p_0<q_0$,
  $p_+=\big(\frac{1}{p_0}-\frac{1}{q_0}\big)^{-1}$, then we get an
  off-diagonal extrapolation theorem  due to Harboure, Macias and
  Segovia~\cite{harboure-macias-segovia88}.   This result allows one
  to extrapolate inequalities of the form~\eqref{eqn:MW1} using
  weights in $A_{p,q}$.  To see that these are equivalent, note that
  by our assumptions and Proposition~\ref{prop:Ap-RHs}, $w\in A_{p,q}$
  if and only if $w^{p} \in A_p \cap RH_{q/p}= A_{p/p_-}\cap
  RH_{(p_+/p)'}$.  

\item If we take $0<p_-<p_+<\infty$ and $p_0=q_0$, we get the limited
  range extrapolation theorem of Auscher and
  Martell, Theorem~\ref{thm:restricted-range}.

\item If we take $p_-=0$, $p_+=1$ and $q_0=p_0$, we get the extrapolation
  theorem for reverse H\"older weights  discovered independently by
  Martell and Prisuelos~\cite{martell-prisuelos}
  and~\cite{Anderson2017Extrapolation-i}.  In the first reference this
  was used to proved weighted norm inequalities for conical square
  functions associated with elliptic operators, and in the second to
  prove weighted norm inequalities for the bilinear fractional
  integral operator.
\end{itemize}

We also note that there is significant overlap
between Theorem~\ref{thm:restricted-offdiag} and an off-diagonal
extrapolation theorem due to Duoandikoetxea~\cite{duoandikoetxeaP}.

\begin{theorem} \label{thm:off-diag-duo}
Given a family of extrapolation pairs $\F$, suppose that for some
$1\leq p_0<\infty$, $0<q_0,\,r_0<\infty$, and $w\in A_{p_0,r_0}$,
inequality~\eqref{eqn:restricted-offdiag1} holds.  Then for all
$1<p<\infty$ and $0<q,\,r<\infty$ such that
\[ \frac{1}{q}-\frac{1}{q_0} = \frac{1}{r}-\frac{1}{r_0} =
  \frac{1}{p}-\frac{1}{p_0}, \]
 and all $w\in A_{p,r}$, inequality~\eqref{eqn:restricted-offdiag2}
 holds.  
\end{theorem}

Note that in the statement of Theorem~\ref{thm:off-diag-duo}, unlike
in the classical definition~\eqref{eqn:Apq-cond}, we do not assume
$p_0<r_0$ or $p<r$.  

If we assume that $r_0\geq \min(p_0,q_0)$, then
Theorem~\ref{thm:off-diag-duo} can be gotten from
Theorem~\ref{thm:restricted-offdiag} by taking $p_-=1$ and $p_+=
\big(\frac{1}{p_0}-\frac{1}{r_0}\big)^{-1}$.  For in this case, by
Proposition~\ref{prop:Ap-RHs} we have that $w\in A_{p_0,r_0}$ is
equivalent to $w^{p_0} \in A_{p_0}\cap RH_{r_0/p_0} = A_{p_0/p_-}\cap
RH_{(p_+/p_0)'}$, and we have
$\frac{1}{q_0}-\frac{1}{p_0}+\frac{1}{p_+}\geq 0$, since $r_0\geq
q_0$. 

Despite this overlap, there are differences between these two theorems.  In
Theorem~\ref{thm:restricted-offdiag} we eliminate the restriction
$p_0,\,p>1$.  And, for values of $p_-\neq 1$, it is not clear whether
Theorem~\ref{thm:restricted-offdiag} can be gotten from
Theorem~\ref{thm:off-diag-duo} by rescaling.  On the other hand,
Theorem~\ref{thm:restricted-offdiag} does not seem to imply
Theorem~\ref{thm:off-diag-duo}  when $r_0<\min(p_0,q_0)$.  

\section{Bilinear extrapolation}
\label{section:bilinear}

In this section we introduce bilinear extrapolation and  show how
Theorem~\ref{thm:restricted-offdiag} can be used to prove it.  All of
the results we consider in this section are true in the multilinear
case, but we restrict ourselves to bilinear inequalities to simplify
the presentation.  

We are interested in weighted, bilinear inequalities of the form
\begin{equation} \label{eqn:bilinear}
 \|T(f,g)\|_{L^p(w^p)} \lesssim 
\|f\|_{L^{p_1}(w_1^{p_1})}\|g\|_{L^{p_2}(w_2^{p_2})}, 
\end{equation}
where $1<p_1,\,p_2<\infty$, 
$\frac{1}{p}=\frac{1}{p_1}+\frac{1}{p_2}$, and $w=w_1w_2$.\footnote{It is also possible to
consider endpoint inequalities where $p_1=1$ or $p_2=1$ and we replace
$L^p(w^p)$ by $L^{p,\infty}(w^p)$, but for brevity we will not consider
this case.}  Weighted norm inequalities of this kind were first
considered by Grafakos and Torres~\cite{MR1947875} and Grafakos and
Martell~\cite{MR2030573} for bilinear Calder\'on-Zygmund singular integrals.   Lerner, {\em et al.}~\cite{MR2483720} introduced a generalization of the Muckenhoupt $A_p$ condition.
Given $\vec{p}=(p_1,p_2,p)$, $\vec{w}=(w_1,w_2,w)$, we say $\vec{w}\in
A_{\vec{p}}$ if 
\[ [\vec{w}]_{A_{\vec{p}}} = \sup_Q 
\left(\avgint_Q w^p\,dx\right)^{\frac{1}{p}}
\left(\avgint_Q w_1^{-p_1'}\,dx\right)^{\frac{1}{p_1'}}
\left(\avgint_Q w_2^{-p_2'}\,dx\right)^{\frac{1}{p_2'}}
< \infty. \]
They showed that a necessary and sufficient condition for the bilinear
maximal operator\footnote{Properly, this operator should be
called the ``bi-sublinear'' maximal operator, but it is common  to
abuse terminology and simply refer to it as a bilinear operator.} 
\[ M(f,g)(x) = \sup_Q \avgint_Q |f|\,dy\, \avgint_Q |g|\,dy \cdot
  \chi_Q(x), \]
is that $\vec{w}\in A_{\vec{p}}$.
Using this fact they showed that the $A_{\vec{p}}$ condition is sufficient
for a bilinear Calder\'on-Zygmund singular integral operator $T$ to
satisfy \eqref{eqn:bilinear}.    (For the precise definition of these
operators and their unweighted theory, see~\cite{MR1880324}.)

It is natural to expect that there is a bilinear extrapolation theory
for weights in $A_{\vec{p}}$, but it is unknown whether this is
possible.   This remains a very important open question in the theory
of bilinear weighted norm inequalities. 

\medskip

Therefore, to develop a theory of extrapolation we will work with a
restricted class of weights $\vec{w}$ where $w_i^{p_i} \in A_{p_i}$,
$i=1,2$.  By H\"older's inequality, we have that in this case
$\vec{w}\in A_{\vec{p}}$, but it is relatively straightforward to
construct examples of weights in $A_{\vec{p}}$ such that $w_i^{p_i}
\not\in A_{p_i}$:  see~\cite{DCU-KM-2017,MR2483720}.

We generalize the formalism of extrapolation pairs to the bilinear
setting by defining a family $\F$ of extrapolation triples:  $(f,g,h)$
such that each function is non-negative, measurable, and not
identically $0$.  If we write
\[ \|h\|_{L^p(w^p)} \lesssim 
\|f\|_{L^{p_1}(w_1^{p_1})}\|g\|_{L^{p_2}(w_2^{p_2})}, \qquad (f,g,h) \in \F, \]
then we mean that this inequality holds for every triple in $\F$ such
that $\|h\|_{L^p(w^p)}<\infty$.  As in the linear case, it is
straightforward to prove weighted norm inequalities for operators:
the ideas in Section~\ref{section:examples} extend immediately to the
bilinear setting.  Similarly, we can use extrapolation to prove
bilinear versions of Corollaries~\ref{cor:weak-rdf}
and~\ref{cor:vv-rdf}.

Bilinear extrapolation was first proved for operators by Grafakos and Martell~\cite{MR2030573};
the following theorem generalizes their result to families of
extrapolation triples.

\begin{theorem} \label{thm:bilinear-extrapol}
Given a family $\F$ of extrapolation triples, suppose that for some
$\vec{p}=(p_1,p_2,p)$, where $1\leq p_1,\,p_2<\infty$ and
$\frac{1}{p}=\frac{1}{p_1}+\frac{1}{p_2}$, and weights
$\vec{w}=(w_1,w_2,w)$ such that $w_i^{p_i}\in A_{p_i}$ and $w=w_1w_2$, 
\[ \|h\|_{L^p(w^p)} \lesssim 
\|f\|_{L^{p_1}(w_1^{p_1})}\|g\|_{L^{p_2}(w_2^{p_2})}, \qquad  (f,g,h) \in \F. \]
Then for every $\vec{q}=(q_1,q_2,q)$, where $1< q_1,\,q_2<\infty$ and
$\frac{1}{q}=\frac{1}{q_1}+\frac{1}{q_2}$, and weights
$\vec{w}=(w_1,w_2,w)$ such that $w_i^{q_i}\in A_{q_i}$ and $w=w_1w_2$,
\[ \|h\|_{L^q(w^q)} \lesssim 
\|f\|_{L^{q_1}(w_1^{q_1})}\|g\|_{L^{q_2}(w_2^{q_2})}, \qquad (f,g,h) \in \F. \]
\end{theorem}

We can also prove a restricted range version of
Theorem~\ref{thm:bilinear-extrapol}, but in order to make the main
ideas of the proof clearer, we omit this generalization.  For details,
see~\cite{CruzUribe:2017vx}. 

\begin{proof}
Our proof is adapted from Duoandikoetxea~\cite{duoandikoetxeaP}.  Given
$1<p_1,\,p_2<\infty$,  fix $w_2^{p_2}\in A_{p_2}$ and fix a
function $g$ such that there exist functions $f,\,h$ with $(f,g,h)\in
\F$.  By assumption $\|g\|_{L^{p_2}(w^{p_2})}>0$; assume for the moment that
$\|g\|_{L^{p_2}(w^{p_2})}<\infty$.  Define a new family of extrapolation pairs
\[ \F_g  = \big\{ (F,f) = \big(hw_2\|g\|_{L^{p_2}(w_2^{p_2})}^{-1},f) :
  (f,g,h) \in \F\big\}.  \]
Let $p_1=1$, $p_+=\infty$.  Since $\|F\|_{L^{p}(w^{p})}<\infty$ if and only if
  $\|h\|_{L^p(w^p)}<\infty$, for all $w_1$ such that $w_1^{p_1}\in A_{p_1}=
  A_{p_1/p_-} \cap RH_{(p_+/p_1)'}$, 
\[ \|F\|_{L^p(w_1^p)} \lesssim \|f\|_{L^{p_1}(w_1^{p_1})},  \qquad
  (F,f) \in \F_g. \]
Moreover, we have that $\frac{1}{p}-\frac{1}{p_1}+\frac{1}{p_+} \geq
0$ since $\frac{1}{p}>\frac{1}{p_1}$.  Therefore, by
Theorem~\ref{thm:restricted-offdiag}, for all $q$ and $q_1$ such that
\begin{equation} \label{eqn:q-cond}
\frac{1}{q}-\frac{1}{q_1} = \frac{1}{p}-\frac{1}{p_1}, 
\end{equation}
and all $w_1$ such that $w_1^{q_1} \in A_{q_1}$, 
\[ \|F\|_{L^q(w_1^q)} \lesssim \|f\|_{L^{q_1}(w_1^{q_1})},
\qquad (F,f) \in \F_g. \]
By the definition of $\F_g$ we therefore have that
\begin{equation} \label{eqn:bilinear-intermed}
\|h\|_{L^q(w^q)} \leq \|f\|_{L^{q_1}(w_1^{q_1})}\|g\|_{L^{p_2}(w_2^{p_2})},
\end{equation}
provided that we assume that $\|g\|_{L^{p_2}(w_2^{p_2})}<\infty$.
However, if $\|g\|_{L^{p_2}(w_2^{p_2})}=\infty$, then
inequality~\eqref{eqn:bilinear-intermed} still holds.  
Since this is true for all  $g$ and $w_2$ with
$w_2^{p_2}\in A_{p_2}$, we must have that \eqref{eqn:bilinear-intermed}
holds for all $(f,g,h)\in \F$.  Furthermore,
note that \eqref{eqn:q-cond} implies that 
\[ \frac{1}{q} = \frac{1}{q_1} +
  \frac{1}{p_1}+\frac{1}{p_2}-\frac{1}{p_1}
= \frac{1}{q_1}+\frac{1}{p_2}.  \]

We now  repeat this argument:  fix $q$ and $q_1$ such that
$\frac{1}{q} = \frac{1}{q_1}+\frac{1}{p_2}$ and weight $w_1$ such that
$w_1^{q_1}\in A_{q_1}$.  Fix a function $f$ such that
$0<\|f\|_{L^{q_1}(w_1^{q_1})}<\infty$ and there exist
$g,\,h$ with $(f,g,h)\in \F$.   Define the new family
\[ \F_f = \big\{ (G,g) = (hw_1 \|f\|_{L^{q_1}(w_1^{q_1})}^{-1}, g) :
  (f,g,h) \in \F \big\}.  \]
Then we can argue as above, applying Theorem~\ref{thm:restricted-offdiag} to
conclude that for all $1<q_1,q_2<\infty$ and $w_i^{q_i}\in A_{q_i}$,
$i=1,2$, 
\[ \|h\|_{L^q(w^q)} \lesssim \|f\|_{L^{q_1}(w_1^{q_1})}
\|g\|_{L^{q_2}(w_2^{q_2})}, \qquad (f,g,h) \in \F. \]
\end{proof}

\bigskip

As an application of Theorem~\ref{thm:bilinear-extrapol} we give an
elementary proof of  weighted norm inequalities for bilinear
Calder\'on-Zygmund singular integral operators for this restricted
class of weights.   

\begin{theorem}  \label{thm:bilinear-czo}
Let $T$ be a bilinear Calder\'on-Zygmund singular integral operator.
Then for all $1<p_1,\,p_2<\infty$,
$\frac{1}{p}=\frac{1}{p_1}+\frac{1}{p_2}$, and weights $w_i^{p_i} \in
A_{p_i}$, $i=1,2$, $w=w_1w_2$,
\[ \|T(f,g)\|_{L^p(w^p)} \lesssim 
\|f\|_{L^{p_1}(w_1^{p_1})}\|g\|_{L^{p_2}(w_2^{p_2})}. \]
\end{theorem}

\begin{proof}
Again, we use domination by sparse operators.  If
$T$ is a bilinear singular integral and $f,\,g\in L^\infty_c$, then,
with the same notation for dyadic grids used in
Section~\ref{section:examples}, there exist sparse sets $\Ss_t$
such that 
\begin{equation} \label{eqn:bilinear-sparse}
 |T(f,g)(x)| \lesssim \sum_{t\in \{0,\pm\frac{1}{3}\}^n}
T_{\Ss_t}(|f|,|g|), 
\end{equation}
where for any sparse set $\Ss$,
\[ T_{\Ss}(f,g)(x) = \sum_{Q\in \Ss} \avgint_Q f\,dy \,\avgint_Q g\,dy
  \cdot \chi_Q(x). \]

By Theorem~\ref{thm:bilinear-extrapol} it will suffice to show that
given any dyadic grid $\D^t$, sparse set $\Ss\subset \D^t$, and
weights $w_i$ such that $w_i^2\in A_2$, $i=1,2$, we have that for all
non-negative functions $f,\,g\in L^\infty_c$,
\[ \|T_\Ss (f,g)\|_{L^1(w)}
\lesssim \|f\|_{L^2(w_1^2)}\|g\|_{L^2(w_2^2)}. \]
The proof is nearly identical to the argument in the linear case given in
the proof of Theorem~\ref{thm:A2-conj}.  Let $\sigma_i=w_i^{-2}$, $i=1,2$.
Then 
\begin{align*}
\|T_\Ss (f,g)\|_{L^1(w)}
& = \sum_{Q\in\Ss} \avgint_Q f\,dx\, \avgint_Q g\,dx \, \avgint_Q
  w\,dx\,  |Q|
  \\
& = \sum_{Q\in\Ss} \avgint_Q w_1w_2\,dx \,
\avgint_Q \sigma_1\,dx \, \avgint_Q \sigma_2\,dx \,
\avgint_Q fw_1^2\, d\sigma_1 \, \avgint_Q g w_2^2\,d\sigma_2 \,|Q|.
\end{align*}

By assumption, $w_i^{-2} \in A_2$, and so by
Proposition~\ref{prop:Ap-RHs}, $w_i \in A_2\cap RH_2$.  Therefore, by
H\"older's inequality and 
Proposition~\ref{prop:multilinear-rhi}, 
\begin{multline*}
 \avgint_Q w_1w_2\,dx \avgint_Q w_1^{-2}\,dx 
\avgint_Q w_2^{-2}\,dx  \\
\lesssim 
\left(\avgint_Q w_1^{2}\,dx\right)^{\frac{1}{2}}
\left(\avgint_Q w_2^{2}\,dx\right)^{\frac{1}{2}}
\left(\avgint_Q w_1^{-2}\,dx\right)^{\frac{1}{2}}
\left(\avgint_Q w_2^{-2}\,dx\right)^{\frac{1}{2}}
\avgint_Q w_1^{-1}w_2^{-1}\,dx \\
\lesssim
\avgint_Q w_1^{-1}w_2^{-1}\,dx
\lesssim
\frac{1}{|Q|}\int_{E_Q} w_1^{-1}w_2^{-1}\,dx.
\end{multline*}
The last two inequalities hold since $w_1^{-1}w_2^{-1}\in A_2$:  this in
turn follows from H\"older's inequality since $w_i^{-2}\in A_2$,
$i=1,2$.  The final inequality then follows from
Lemma~\ref{lemma:Ap-prop}. 

Hence, we can continue the above estimate, getting
\begin{align*}
& \sum_{Q\in\Ss} \avgint_Q w_1w_2\,dx \,
\avgint_Q \sigma_1\,dx \, \avgint_Q \sigma_2\,dx \,
\avgint_Q fw_1^2\, d\sigma_1 \, \avgint_Q g w_2^2\,d\sigma_2 \, |Q| \\
& \qquad \qquad \lesssim
\sum_{Q\in\Ss} 
\avgint_Q fw_1^2\, d\sigma_1 \, \avgint_Q g w_2^2\,d\sigma_2 \,
\int_{E_Q} \sigma_1^{\frac{1}{2}}\sigma_2^{\frac{1}{2}}\,dx \\
& \qquad \qquad \leq
\int_\subRn M^{\D^t}_{\sigma_1}(fw_1^2) M^{\D^t}_{\sigma_2}(gw_2^2)
\sigma_1^{\frac{1}{2}}\sigma_2^{\frac{1}{2}}\,dx \\
& \qquad \qquad \leq 
\|M^{\D^t}_{\sigma_1}(fw_1^2)\|_{L^2(\sigma_1)}
\|M^{\D^t}_{\sigma_2}(gw_2^2)\|_{L^2(\sigma_2)}; \\
\intertext{by Lemma~\ref{lemma:wtd-max-op}, which holds for arbitrary
  dyadic grids,}
& \qquad \qquad \lesssim
\|fw_1^2\|_{L^2(\sigma_1)}
\|gw_2^2\|_{L^2(\sigma_2)} \\
& \qquad \qquad = 
\|f\|_{L^2(w_1^2)}\|g\|_{L^2(w_2^2)}. 
\end{align*}
\end{proof}

\section{Extrapolation on Banach function spaces}
\label{section:BFS}

In this final section we  discuss how extrapolation can be used to
prove norm inequalities in Banach function spaces, starting from norm
inequalities in weighted $L^p$.  This lets us generalize the aphorism
of Antonio C\'ordoba given in Section~\ref{section:introduction} and
assert: ``{\em There are no Banach function spaces, only weighted
  $L^2$.}'' (Cf.~\cite[Chapter~1]{MR2797562}.)

We begin with some definitions.  For more information on the theory of Banach function
spaces, see Bennett and Sharpley~\cite{bennett-sharpley88}.   Let $\X$ 
be a Banach space of Lebesgue measurable functions defined on $\R^n$
with norm $\|\cdot\|_\X$.
We say that $\X$ is a Banach function space if the norm satisfies the
following properties:
\begin{itemize}
\item if $|f|\leq |g|$ a.e., then $\|f\|_\X\leq \|g\|_\X$;

\item if $|f_k|$ increases pointwise a.e. to $|f|$, then 
$\|f_k\|_\X\rightarrow \|f\|_\X$;

\item if $E\subset \R^n$, $|E|<\infty$, then $\|\chi_E\|_\X<\infty$,
  and there exists $C(E)>0$ such that for all $f\in \X$,
\[ \int_E|f|\,dx \leq C(E)\|f\|_\X. \]
\end{itemize}

Given a Banach function space $\X$, we define the associate space
$\X'$ to be the set of measurable functions $g$ such that 
\[ \|g\|_{\X'} = \sup\left\{ \int_\subRn fg\,dx : \|f\|_\X \leq 1
  \right\} < \infty.  \]
Then $\|\cdot\|_{\X'}$ is a norm and $\X'$ is itself a Banach function
space.    The two norms are related by the generalized H\"older's
inequality, 
\[ \int_\subRn |fg|\,dx \leq \|f\|_\X \|g\|_{\X'}. \]
The associate space embeds (up to an isomorphism) in the dual
space $\X^*$, and in some (though not all) cases they are equal.
However, the associate spaces are always reflexive:  for any Banach
function space $\X$, $(\X')' = \X$. 

\begin{theorem} \label{thm:BFS-extrapol}
Given a family of extrapolation pairs $\F$, suppose that for some
$1\leq p_0<\infty$ and every $w_0\in A_{p_0}$,
\begin{equation} \label{eqn:BFS-extrapol1}
\int_\subRn f^{p_0}w_0\,dx \lesssim 
\int_\subRn g^{p_0}w_0\,dx, \qquad (f,g) \in \F. 
\end{equation}
Let $\X$ be a Banach function space such that the maximal operator
satisfies $M : \X \rightarrow \X$ and $M : \X'\rightarrow \X'$.  Then
\begin{equation} \label{eqn:BFS-extrapol2}
 \|f\|_\X \lesssim \|g\|_\X, \qquad (f,g) \in \F. 
\end{equation}
\end{theorem}

\begin{proof}
The proof is actually a simple variation of the proof of
Theorem~\ref{thm:rdf-extrapolation}.   By this result, we may assume
without loss of generality that \eqref{eqn:BFS-extrapol1} holds for
$p_0=2$ and weights $w_0\in A_2$.   We define two iteration algorithms:
\[ 
 \Rh_1 h_1 = \sum_{k=0}^\infty \frac{M^k h_1}{2^k \|M\|_\X^k}, \qquad
 \Rh_2 h_2 = \sum_{k=0}^\infty \frac{M^k h_2}{2^k \|M\|_{\X'}^k}.
\]
Then the proof of Theorem~\ref{thm:rdf-algorithm} generalizes to give
the following: \\

\begin{tabular}{l l  l l}
 ($A_1$) & $h_1(x)\le \Rh_1 h_1(x)$
& \qquad
($A_2$) & $h_2(x)\le \Rh_2 h_2(x)$ \\[6pt]
($B_1$) & $\|\Rh_1 h_1\|_{\X}\le 2\|h_1\|_{\X}$
& \qquad
($B_2$) & $\|\Rh_2 h_2\|_{\X'}\le 2\|h_2\|_{\X'}$
\\[6pt] 
($C_1$) & $[\Rh_1 h_1]_{A_{1}}\le 2\|M\|_{\X}$
& \qquad 
($C_2$) & $[\Rh_2 h_2]_{A_{1}}\le 2\|M\|_{\X'}$.
\end{tabular} \\

\medskip

Now fix $(f,g)\in \F$; without loss of generality $0<\|f\|_\X, \,
\|g\|_\X<\infty$.  Define
\[ h_1 = \frac{f}{\|f\|_\X}+\frac{g}{\|g\|_\X}; \]
then $\|h_1\|_\X \leq 2$.   By the definition of the associate
space and reflexivity,  there exists $h_2\in \X'$, $\|h_2\|_{\X'}=1$,
such that 
\begin{align*}
\|f\|_\X 
& \lesssim \int_\subRn fh_2\,dx; \\
\intertext{by $(A_2)$ and H\"older's inequality,}
& \leq \int_\subRn f (\Rh_1 h_1)^{-\frac{1}{2}}
(\Rh_1 h_1)^{\frac{1}{2}} \Rh_2h_2\,dx  \\
& \leq \left(\int_\subRn f^2 (\Rh_1
  h_1)^{-1}\Rh_2h_2\,dx\right)^{\frac{1}{2}} 
\left(\int_\subRn \Rh_1 h_1 \Rh_2 h_2 \,dx\right) ^{\frac{1}{2}} \\
& = I_1 ^{\frac{1}{2}} \cdot I_2 ^{\frac{1}{2}}. 
\end{align*}

To estimate $I_2$ we use the generalized H\"older's inequality and
$(B_1)$ and $(B_2)$:
\[ I_2 \leq \|\Rh_1 h_1\|_\X \|\Rh_2 h_2 \|_{\X'}
\leq 4\|h_1\|_\X \|h_2 \|_{\X'}  \leq 8. \]
To estimate $I_1$, note first that by $(A_1)$, $f \leq h_1\|f\|_\X \leq \|f\|_\X \Rh_1
h_1$, so $I_1 \leq \|f\|_\X^2 I_2 < \infty$.   Furthermore, by
$(C_1)$, $(C_2)$ and Theorem~\ref{thm:jones}, $(\Rh_1
  h_1)^{-1}\Rh_2h_2\in A_2$.  Therefore, by~\eqref{eqn:BFS-extrapol1}
 and again by $(A_1)$, 
\[ I_1 \lesssim \int_\subRn g^2 (\Rh_1
  h_1)^{-1}\Rh_2h_2\,dx 
\leq \|g\|_\X^2 \cdot I_2 \lesssim \|g\|_\X^2.  \]
If we combine all of these inequalities we
get~\eqref{eqn:BFS-extrapol2} and our proof is complete.
\end{proof}

\medskip

Extrapolation into Banach function spaces was first considered
in~\cite{cruz-uribe-fiorenza-martell-perez06} in the context of the
variable Lebesgue spaces (see below).  The result proved there is
somewhat different, and only requires that~\eqref{eqn:BFS-extrapol1}
holds for weights $w_0\in A_1$, though a version of
Theorem~\ref{thm:BFS-extrapol} was proved as a corollary.   Theorem~\ref{thm:BFS-extrapol} is a
variant of the extrapolation theorem proved for the weighted
variable Lebesgue spaces in~\cite{CruzUribeSFO:2017km}.   Curbera,
{\em et al.}~\cite{curbera-garcia-cuerva-martell-perez06} proved an
extrapolation theorem into rearrangement invariant Banach
function spaces such as Orlicz spaces.   For a general
treatment of extrapolation into Banach function spaces,
see~\cite[Chapter~4]{MR2797562}.   Very recently
in~\cite{Cruz-Uribe2016Extrapolation-a}, extrapolation was extended to
the Musielak-Orlicz spaces, a very general class of function spaces
that include the Lebesgue spaces, Orlicz spaces, and the variable
Lebesgue spaces as special cases.  (For more information about these spaces,
see~\cite{musielak83}.)

\medskip

We conclude these notes by considering the application of
extrapolation to the variable Lebesgue spaces.  These spaces are a
generalization of the classical Lebesgue spaces, replacing the
constant exponent $p$ with an exponent function $\pp$.  We begin with some
definitions; for complete details and references on these spaces,
see~\cite{cruz-fiorenza-book,MR2790542}.   Given a measurable function $\pp :
\R^n \rightarrow [1,\infty]$, let $\R^n_\infty= \{ x \in \R^n :
p(x)=\infty\}$, and define
\[ p_- = \essinf_x p(x), \qquad p_+ = \esssup_x p(x). \]

Define $\Lp$ to be the set of
measurable functions $f$ such that for some $\lambda>0$, 
\[ \rho_\pp(f/\lambda) = 
\int_{\R^n\setminus \R^n_\infty} 
\left(\frac{|f(x)|}{\lambda}\right)^{p(x)}\,dx 
+ \lambda^{-1}\|f\|_{L^\infty(\R^n_\infty)} < \infty. \]
Then $\Lp$ is a Banach function space with respect to the
Luxemburg norm
\[ \|f\|_{\Lp} =\|f\|_\pp 
= \inf\big\{ \lambda > 0 : \rho_\pp (f/\lambda) \leq 1 \big\}. \]
When $\pp=p$, $1\leq p \leq \infty$,  $\Lp=L^p$ with equality of
norms.  

The associate space of $\Lp$ equals $L^\cpp$ with an equivalent norm, where $\cpp$
is defined pointwise by
\[ \frac{1}{p(x)} + \frac{1}{p'(x)} = 1 \]
with the convention $1/\infty=0$.  Consequently, we have the generalized
H\"older's inequality
\[ \int_\subRn |fg|\,dx \leq 2\|f\|_\pp \|g\|_\cpp. \]

The boundedness of the maximal operator on $\Lp$ requires some regularity on the
exponent $\pp$.  A very useful sufficient condition is log-H\"older
continuity, defined locally by
\[ \left|\frac{1}{p(x)}-\frac{1}{p(y)}\right| \leq 
\frac{C_0}{-\log(|x-y|)}, \qquad |x-y| < \frac{1}{2}, \]
and at infinity by
\[ \left|\frac{1}{p(x)}-\frac{1}{p_\infty}\right|
\leq \frac{C_\infty}{\log(e+x|)}. \]
We denote this by writing $\pp \in LH$.  The following result was first proved
in~\cite{cruz-uribe-fiorenza-neugebauer03}; for a simpler proof,
see~\cite[Chapter~3]{cruz-fiorenza-book}. 

\begin{theorem} \label{thm:var-max-op}
Given an exponent function $\pp$ such that $1<p_-\leq p_+\leq \infty$
and such that $\pp \in LH$, 
$ \|Mf \|_\pp \lesssim \|f\|_\pp$. 
\end{theorem}

Clearly, if $\pp \in LH$,
then $\cpp \in LH$, so if $1<p_-\leq p_+<\infty$ and $\pp \in LH$, the
maximal operator is bounded on $\Lp$ and $L^\cpp$.  Moreover,
Diening~\cite{diening05,MR2790542}
proved the following very deep result:  given any exponent function
$\pp$, if $1<p_-\leq p_+<\infty$, the
maximal operator is bounded on $\Lp$ if and only if it is bounded on
$L^\cpp$.\footnote{Diening also showed that if $M$ is bounded on $\Lp$, then there
exists $s>1$ such that it is also bounded on $L^{\pp/s}$. 
We used this fact instead of the boundedness of $M$ on both $\Lp$ and $L^\cpp$ to prove
our extrapolation theorem
in~\cite{cruz-uribe-fiorenza-martell-perez06}.  In addition, we
  assumed an abstract version of this property to prove extrapolation for
  general Banach function spaces in~\cite{MR2797562}.}

It follows from these facts that we can apply extrapolation to the
variable Lebesgue spaces $\Lp$, assuming only that $1<p_-\leq
p_+<\infty$ and that the
maximal operator is bounded on $\Lp$.   As an immediate consequence,
we get that in this case, if $T$ is a Calder\'on-Zygmund singular
integral, then $\|Tf\|_\pp \lesssim \|f\|_\pp$ whenever $M$ is bounded
on $\Lp$.    In~\cite{cruz-fiorenza-book} we conjectured that this was
a necessary as well as sufficient condition.  We recently learned that
this conjecture was proved by Rutsky~\cite{rutskyP}.  

Similarly, many other norm inequalities can be extended to variable
Lebesgue spaces using the corresponding weighted norm inequalities.
For a number of examples,
see~\cite{cruz-fiorenza-book,cruz-uribe-fiorenza-martell-perez06,MR2797562}.
For the application of extrapolation to develop the theory of variable
Hardy spaces, see~\cite{DCU-dw-P2014}.  Finally,
in~\cite{CruzUribe:2016wv} we developed a theory of bilinear
extrapolation which we used to prove bilinear inequalities in variable
Lebesgue spaces starting from weighted bilinear inequalities. This led
to both new (and simpler) proofs of known results for bilinear
operators on variable Lebesgue spaces and also to new results.  

\bibliographystyle{plain}
\bibliography{paseky-lecture}

\end{document}